\documentclass[a4paper]{amsart}

\usepackage{a4wide,amsfonts,amssymb,amsmath,color,todonotes,tikz-cd,mathtools,graphicx}
\usepackage[english]{babel}
\usepackage[style=ext-numeric,
 maxbibnames=99,   
 giveninits=false, 
 sortcites=true,   
 backend=biber]{biblatex}
\allowdisplaybreaks

\newtheorem{lem}{Lemma}[section]

\newtheorem{theo}[lem]{Theorem}

\newtheorem{cor}[lem]{Corollary}
\newtheorem{rem}[lem]{Remark}

\newtheorem{ex}[lem]{Example}
\def\ol{\overline}
\def\ot{\leftarrow}

\def\wto{\xrightharpoonup}
\def\cpt{\hookrightarrow}
\def\equi{\Leftrightarrow}
\def\qequi{\quad\equi\quad}
\def\qqequi{\quad\qequi\quad}

\newcommand{\wh}[1]{\widehat{#1}}
\newcommand{\wt}[1]{\widetilde{#1}}
\DeclareUnicodeCharacter{25C7}{\bigdiamond}

\def\n{\mathbb{N}}

\def\reals{\mathbb{R}}

\def\rt{\reals^{3}}

\def\om{\Omega}
\def\ga{\Gamma}
\def\eps{\varepsilon}

\def\cg{c_{\mathsf{g}}}
\def\cf{c_{\mathsf{f}}}
\def\nor{\mathsf{n}}
\def\tan{\mathsf{t}}
\def\dir{\mathsf{D}}
\def\neu{\mathsf{N}}

\def\sfL{\mathsf{L}}
\def\sfH{\mathsf{H}}
\def\sfC{\mathsf{C}}
\def\calH{\mathcal{H}}
\def\calO{\mathcal{O}}
\def\cI{\mathcal{I}}
\def\sfD{\mathsf{D}}
\def\sfN{\mathsf{N}}

\renewcommand{\L}{\sfL}
\newcommand{\Lt}{\L^{2}}
\renewcommand{\H}{\sfH}
\newcommand{\Hc}{\mathring{\sfH}}
\newcommand{\Hd}{\dot{\sfH}}
\newcommand{\C}{\sfC}
\newcommand{\Cc}{\mathring{\sfC}}
\newcommand{\Cd}{\dot{\sfC}}
\newcommand{\Ci}{\C^{\infty}}
\newcommand{\Cic}{\Cc^{\infty}}

\DeclareMathOperator{\supp}{supp}

\DeclareMathOperator{\tr}{tr}
\DeclareMathOperator{\trs}{tr_{\mathsf{s}}}
\DeclareMathOperator{\trn}{tr_{\mathsf{n}}}
\DeclareMathOperator{\trt}{tr_{\mathsf{t}}}
\DeclareMathOperator{\trtx}{tr_{\mathsf{tx}}}

\DeclareMathOperator{\dist}{dist}
\DeclareMathOperator{\p}{\partial}
\DeclareMathOperator{\id}{id}
\DeclareMathOperator{\na}{\nabla}
\DeclareMathOperator{\nac}{\mathring{\na}}
\DeclareMathOperator{\rot}{rot}
\DeclareMathOperator{\rotc}{\mathring{\rot}}

\def\div{\operatorname{div}}
\DeclareMathOperator{\divc}{\mathring{\div}}

\DeclareMathOperator{\db}{d}

\def\cD{\operatorname{\mathcal{D}}}
\def\M{\operatorname{M}}
\def\S{\operatorname{S}}
\def\T{\operatorname{T}}
\def\A{\operatorname{A}}
\def\B{\operatorname{B}}

\def\Pot{\operatorname{P}}

\def\adj{\operatorname{adj}}
\def\P{\Pot}

\newcommand{\norm}[1]{\|#1\|}
\newcommand{\bnorm}[1]{\big\|#1\big\|}

\newcommand{\scp}[2]{\langle#1,#2\rangle}
\newcommand{\bscp}[2]{\big\langle#1,#2\big\rangle}
\newcommand{\Bscp}[2]{\Big\langle#1,#2\Big\rangle}
\newcommand{\dual}[2]{\langle\!\langle#1,#2\rangle\!\rangle}
\newcommand{\bdual}[2]{\big\langle\!\big\langle#1,#2\big\rangle\!\big\rangle}
\newcommand{\tvec}[2]{\begin{bmatrix}#1\\#2\end{bmatrix}}
\newcommand{\thvec}[3]{\begin{bmatrix}#1\\#2\\#3\end{bmatrix}}
\def\subset{\subseteq}
\def\Xi{\Theta}

\title[The Closed Range Property of the de Rham Complex in Unbounded Domains]
{The Closed Range Property
and Gaffney's inequality\\
of the de Rham Complex in Unbounded Domains}
\author{Dirk Pauly}
\address{Institut f\"ur Analysis, Technische Universit\"at Dresden, Germany}
\email[Dirk Pauly]{dirk.pauly@tu-dresden.de}
\author{Marcus Waurick}
\address{Institut f\"ur Angewandte Analysis, Technische Universit\"at Bergakademie Freiberg, Germany}
\email[Marcus Waurick]{marcus.waurick@math.tu-freiberg.de}
\keywords{de Rham complex, Maxwell's equations, closed ranges, Friedrichs/Poincar\'e estimates, low frequency asymptotics,
Gaffney estimates, unbounded domains, cubes, cuboids}
\subjclass{47F05, 47F10, 35P05, 35P15, 58J10, 58A12}
\date{\today}
\thanks{\it{Corresponding Author.}~Dirk Pauly}
\dedicatory{Dedicated to our Teacher and Dear Friend\\ 
Rainer Picard (The Captain)\\ 
on the occasion of his 80th birthday}

\begin{document}


\begin{abstract}
The classical Poincar\'e estimate establishes closedness of the range of the gradient in unweighted $\Lt(\om)$-spaces 
as long as $\om\subset\rt$ is contained in a slab, that is, $\om$ is bounded in one direction.  
Here, as a main observation, we provide closed range results for the $\rot$-operator, 
if (and only if) $\om$ is bounded in two directions. 
Along the way, we characterise closed range results for all the differential operators 
of the primal and dual de Rham complex in terms of directions of boundedness of the underlying domain.

As a main application, one obtains the existence of a spectral gap near the $0$ 
of the Maxwell operator allowing for exponential stability results 
for solutions of Maxwell's equations with sufficient damping in the conductivity. 

Our results are based on the validity of Gaffney's (in)equality and the transition 
of the same to unbounded (simple) domains as well as on the stability 
of closed range results under bi-Lipschitz regular transformations. 
The latter technique is well-known and detailed in the appendix; for the results concerning Gaffney's estimate, 
we shall provide accessible, simple proofs using mere standard results.

Moreover, we shall present non-trivial examples and a closed range result for $\rot$ 
with mixed boundary conditions on a set bounded in one direction only.
\end{abstract}


\vspace*{-10ex}
\maketitle


\setcounter{tocdepth}{3}
\vspace*{-5ex}
{\footnotesize
\tableofcontents}


\newpage

\section{Introduction}
\label{sec:intro}

This article is concerned with closed range results for the Maxwell operator given by the block operator matrix 
$$\M=\begin{bmatrix} 0 & -\rot\\\rotc & 0\end{bmatrix}$$
with the two $\rot$-type operators in $\rt$ being endowed with either full homogeneous boundary conditions
($\rotc$) or non at all ($\rot$). These operators are (Hilbert space) adjoints to one another, making $\M$ skew-selfadjoint,
so that it suffices to consider $\rot$-type operators alone for the question of a closed range, since, 
by Banach's closed range theorem, the one ($\rotc$) has closed range if and only if the other ($\rot$) has closed range. 

Quite generally, for closed range results there is a multitude of applications.  
We refer to Section \ref{sec:app} for a more detailed account of the consequences of a closed range. 
In any case, it is well known that a closed range serves as the decisive assumption 
to develop solution theories for partial differential equations,
see, e.g., the main theorem in \cite{TW14} confirming this for 
(possibly nonlinear) elliptic partial differential equations in variational form,
or \cite{P2020a} for linear problems together with functional a posteriori error estimates. 
Moreover, as a closed range constitutes a spectral gap in a punctured neighbourhood of $0$, 
such results can be used to establish exponential stability (see \cite{DTW24}) 
for solutions of time-dependent partial differential equations. 
The mentioned cases highlight Maxwell's equations. 
It is all the more standard to have similar results for the heat equation,
see, e.g., \cite[Section 11.3]{STW22}.

For $\rot$, typically, closed range results can be deduced from certain compact embedding theorems 
and this requires boundedness and sufficient smoothness of the underlying domains. 
For unbounded domains, compact embedding results are, as a rule, not true. 
In order to still obtain closed range results for, e.g., exterior domains (i.e., complements of compacts) 
polynomially weighted Sobolev spaces are introduced, see  \cite{L1986a} or the series of papers \cite{zbMATH03281102}, 
 \cite{SW1983a}, 
 \cite{zbMATH04210630,zbMATH00147080,zbMATH00147438,zbMATH00682629,zbMATH01089055,zbMATH01089056}, 
 \cite{PWW2001a},
 \cite{P2006a,P2007a,P2008a,P2008b,P2012a} 
for a non-exhaustive list.

Hence, for unbounded domains, a different strategy needs to be coined if the considered space are to remain unweighted $\Lt$-spaces. 
The hope for closed range results for differential operators on unbounded domains in unweighted spaces 
is not completely unfounded as the classical Poincar\'e-inequality proves closed range results 
for the gradient with Dirichlet boundary conditions on domains being bounded only in one direction. 
Also, particularly for the $\rot$-operator, \cite[Example 6.5]{W25} shows that there exists a class of examples 
of unbounded domains with closed range for the $\rot$-operator, even though this very example 
is a mere rather elementary consequence of the bounded domain case. Quite dramatically, 
this viewpoint has changed by the impressive \cite[Example 10]{ABMW2019a}, which motivated the present research. 
Indeed, this example is the first we became aware of with a genuinely unbounded domain 
in that the closed range result cannot be easily deduced by facts from the bounded domain case. 
In fact, to obtain closedness of the range of $\rot$ explicit computations using
the (discrete) Fourier-transformation were employed. Naturally, the used technique 
is restricted to the geometry of the considered infinite cylinder with bounded rectangular cross-section (or slight variations thereof). 
In the present article we shall endeavour on the path establishing closed range results 
for less particular geometric set-ups fostering more general functional analytic arguments.

Before we dive into the particulars, 
we take an abstract viewpoint and provide the simple functional analytic background for the theory to follow. 
Here and throughout the paper, we denote 
the domain of definition, kernel, and range of a linear operator $\A$
by $D(\A)$, $N(\A)$, and $R(\A)$, respectively. 
Moreover, $\bot\coloneqq\bot_{\H}$ denotes the orthogonal complement in a Hilbert space $\H$.

Let $\H_{0}$ and $\H_{1}$ be Hilbert spaces, and let 
$$\A:D(\A)\subset\H_{0}\to\H_{1}$$
be a densely defined and closed linear operator. 
The basis of our results is the following observation, which itself can be proved using the 
\emph{closed graph theorem}.

\begin{theo}[characterisation of a closed range]
\label{theo:fatblem1}
The following conditions are equivalent:
\begin{itemize}
\item[\bf(i)] 
$R(\A)\subset\H_{1}$ closed
\hfill{\rm\bf(closed range)}
\item[\bf(ii)] 
$\exists\,c_{\A}>0
\quad
\forall\,x\in D({\A})\cap N(\A)^\bot
\qquad
\norm{x}_{\H_{0}}\leq c_{\A}\norm{\A x}_{\H_{1}}$
\hfill{\rm\bf(closed range inequality)}
\end{itemize}
\end{theo}

Thus, in order to establish (i),
one may provide a proof for the \emph{closed range inequality}, 
i.e., the \emph{Friedrichs/Poincar\'e type estimate} (ii).
The argument employing compact embedding results and, thus, asking for the underlying domain to be bounded, 
establishes (ii) using a contradiction argument, see, e.g., 
 \cite[FA-ToolBox]{P2020a} and \cite{P2015b,P2019a,P2019b,PW2021a,DTW24}. 
Thus, for unbounded domains, (ii) from the above theorem needs to be addressed in a more direct way. 

For this, and throughout this paper, 
we turn to the following more specific setting 
and let $\om\subset\rt$ be an open set. Note that $\om$ might not be connected.
We recall the operators
\begin{align*}
\na:D(\na)\subset\Lt(\om)&\to\Lt(\om)^{3};
&
u\mapsto&\na u,\\
\rot:D(\rot)\subset\Lt(\om)^{3}&\to\Lt(\om)^{3};
&
E\mapsto&\na\times\,E,\\
\div:D(\div)\subset\Lt(\om)^{3}&\to\Lt(\om);
&
E\mapsto&\na\cdot\,E
\end{align*}
with domains of definition
\begin{align*}
D(\na)&\coloneqq\big\{u\in\Lt(\om)\colon\na u\in\Lt(\om)^{3}\big\},\\
D(\rot)&\coloneqq\big\{E\in\Lt(\om)^{3}\colon\rot E\in\Lt(\om)^{3}\big\},\\
D(\div)&\coloneqq\big\{E\in\Lt(\om)^{3}\colon\div E\in\Lt(\om)\big\}
\end{align*}
as the maximal $\Lt$-realisations of $\na$, $\rot$, and $\div$,
where formally $\na\coloneqq [\p_{1}\,\p_{2}\,\p_{3}]^{\top}$
denotes the nabla operator/gradient and $\times$ and $\cdot$,
the vector and scalar product of $\rt$, respectively.
We also define the gradient of a vector field $E$ as the transpose of the Jacobian
$$\na E\coloneqq\begin{bmatrix}\na E_{1}&\na E_{2}&\na E_{3}\end{bmatrix}\in\Lt(\om)^{3\times3},
\qquad
E_{j}\in\H^{1}(\om),
\qquad
j\in\{1,2,3\},$$
and write $E\in\H^{1}(\om)$.
From now on, we skip the powers in the notation of the Lebesgue spaces
and simply write $\Lt(\om)$.

It is well-known and, in fact, easy to prove, that $\na$, $\rot$, and $\div$ are closed and densely defined linear operators, 
see \cite{L1986a} for a classical reference or, e.g., \cite[Proposition 6.1.1]{STW22} for a recent source.
We let $\H(\rot,\om)$ be the Hilbert space given by $D(\rot)$ endowed with the respective graph norm. 
In the same way we have $\H^{1}(\om)=\H(\na,\om)$ and $\H(\div,\om)$.
Any closed subspace of $\H^{1}(\om)$, $\H(\rot,\om)$, and $\H(\div,\om)$ 
containing test functions/vector fields (compactly supported and smooth)
describes suitable boundary conditions for $\na$, $\rot$, and $\div$. 
In particular, considering full homogeneous Dirichlet boundary conditions, 
$$\nac\coloneqq-\div^{*},
\qquad
\rotc\coloneqq\rot^{*},
\qquad
\divc\coloneqq-\na^{*}$$
corresponds to scalar, tangential, and normal boundary conditions, respectively. 
Note that 
$$\Hc^{1}(\om)\coloneqq\Hc(\na,\om)\coloneqq D(\nac),
\qquad
\Hc(\rot,\om)\coloneqq D(\rotc),
\qquad
\Hc(\div,\om)\coloneqq D(\divc)$$
are simply the closures of test functions/fields in the respective graph norm.
These operators form the well known primal and dual de Rham Hilbert complex, 
meaning that $R(\nac)\subset N(\rotc)$, $R(\rotc)\subset N(\divc)$ as well as $R(\na)\subset N(\rot)$, $R(\rot)\subset N(\div)$, 
see the classic source \cite{L1986a} or, e.g., \cite[Proposition 6.1.4]{STW22}, and denoted by
\begin{equation}
\label{eq:derham1}
\def\arrowlength{8ex}
\def\arrowdistance{.8}
\begin{tikzcd}[column sep=\arrowlength]
\Lt(\om) 
\ar[r, rightarrow, shift left=\arrowdistance, "\nac"] 
\ar[r, leftarrow, shift right=\arrowdistance, "-\div"']
&
\Lt(\om) 
\ar[r, rightarrow, shift left=\arrowdistance, "\rotc"] 
\ar[r, leftarrow, shift right=\arrowdistance, "\rot"']
& 
\Lt(\om) 
\arrow[r, rightarrow, shift left=\arrowdistance, "\divc"] 
\arrow[r, leftarrow, shift right=\arrowdistance, "-\na"']
& 
\Lt(\om).
\end{tikzcd}
\end{equation}

To indicate the dependence of the underlying domain,
we sometimes use the notations $\rot=\rot_{\om}$ and $\rotc=\rotc_{\om}$
(same for $\na$ and $\div$).


In order to identify a class of (potentially unbounded) domains so that $\rot$ has closed range, 
we appeal to Theorem \ref{theo:fatblem1} from above, however, with a slight detour. 
This detour is \emph{Gaffney's estimate}, which bounds the $\Lt$-norm of the Jacobian 
of a vector field in terms of its $\rot$ and $\div$. 
Only relying on elementary integration by parts, for smooth and compactly supported vector fields $\phi$, one obtains
$$\norm{\na\phi}_{\Lt(\om)}\leq 
\big(\norm{\rot\phi}_{\Lt(\om)}^{2}
+\norm{\div\phi}_{\Lt(\om)}^{2}\big)^{1/2}.$$ 
It is remarkable that, for some $\om$, it is possible to still obtain such an inequality, 
even though the vectors fields do \emph{not} satisfy homogeneous boundary on all of its components. 
For ease of reference, we single out these domains of interest next.

We call an open set $\om$ 
\textbf{Gaffney domain}\footnote{The inequality 
relies on the geometry of $\om$ and, thus, can have the form 
$$\norm{\na E}_{\Lt(\om)}^{2}\leq
\cg^{2}\big(\norm{\rot E}_{\Lt(\om)}^{2} +\norm{\div E}_{\Lt(\om)}^{2}\big)$$
for some $\cg>0$. 
In the present text, we are concerned with $\cg=1$ only, 
so we keep the definition as simple as possible. 
Note that this is for convenience of the reader only. 
The theory to unfold goes through without difficulties also for $\cg\neq1$.}
if for all 
$$E\in\big(D(\rotc)\cap D(\div)\big)\cup\big(D(\rot)\cap D(\divc)\big)$$
both the following conditions hold:
\begin{itemize}
\item[\bf(i)] 
$E\in\H^{1}(\om)$, and
\item[\bf(ii)] 
$\norm{\na E}_{\Lt(\om)}^{2}\leq
\norm{\rot E}_{\Lt(\om)}^{2} +\norm{\div E}_{\Lt(\om)}^{2}$.
\end{itemize}
More particularly, $\om$ is called an \textbf{exact Gaffney domain}, if the inequality sign in (ii) can be replaced by an equality sign.

Hence, in order to establish $\om$ to be a (exact) Gaffney domain 
involves proving an $\H^{1}(\om)$-regularity result (to have (i)) 
and showing Gaffney's estimate to hold for $\rot$- and $\div$-regular vector fields satisfying 
at least one of the two associated homogeneous boundary conditions. 
One can find several publications, 
where the emphasis is put on establishing the 
(by no means trivial) Gaffney's estimate for highly involved geometric set-ups. 
However, in order to establish closed range results for $\rot$ the regularity requirement cannot be neglected. 
Indeed, we may now provide the strategy of how to show closedness of the range in the following:

Let $\om$ be a Gaffney domain. 
In order to identify $\rotc_\om$ having a \emph{closed range} we proceed as follows:
\begin{itemize}
\item[\bf(i)] 
By Theorem \ref{theo:fatblem1} it suffices to prove the corresponding \emph{closed range inequality}, i.e.,
$$\exists c>0\quad
\forall\,E\in D(\rotc)\cap N(\rotc)^\bot
\qquad
\norm{E}_{\Lt(\om)}\leq
c\norm{\rot E}_{\Lt(\om)}.$$
\item[\bf(ii)] 
Since $N(\rotc)^\bot=\overline{R(\rotc^{*})}=\overline{R(\rot)}$, the \emph{complex property} of the de Rham complex yields $N(\rotc)^\bot\subset N(\div)\subset D(\div)$ and, hence,
$$D(\rotc)\cap N(\rotc)^\bot\subset D(\rotc)\cap D(\div).$$
\item[\bf(iii)] 
Since $\om$ is a Gaffney domain, for $E\in D(\rotc)\cap N(\div)$ we have
$$E\in\H^{1}(\om)
\quad\wedge\quad
\norm{\na E}_{\Lt(\om)}\leq\norm{\rot E}_{\Lt(\om)}.$$
\item[\bf(iv)] 
One shows the existence of some $\cf>0$ such that
for $u\in\H^{1}(\om)$ satisfying suitable homogeneous Dirichlet boundary condition 
one has a \emph{Friedrichs' estimate} for the gradient
$$\norm{u}_{\Lt(\om)}\leq
\cf\norm{\na u}_{\Lt(\om)}.$$
\item[\bf(v)] 
One shows that the individual components of 
$E\in\H^{1}(\om)\cap D(\rotc)$ 
satisfy the boundary conditions admissible in (iv).
\item[\bf(vi)] 
By (iii) and (iv), 
one obtains for all $E\in D(\rotc)\cap N(\rotc)^\bot\subset D(\rotc)\cap N(\div)$
$$\norm{E}_{\Lt(\om)}^{2}
=\sum_{j=1}^{3}\norm{E_{j}}_{\Lt(\om)}^{2}\leq 
\cf^{2}\sum_{j=1}^{3}\norm{\na E_{j}}_{\Lt(\om)}^{2}
=\cf^{2}\norm{\na E}_{\Lt(\om)}^{2}
\leq\cf^{2}\norm{\rot E}_{\Lt(\om)}^{2}.$$
\end{itemize}
On a grand scheme of things the structural properties needed for this approach to work
is the Hilbert complex structure, the de Rham complex \eqref{eq:derham1} forms a prominent example of. 
On a technical side, the most demanding part is the establishing of $\om$ to be a Gaffney domain in the first place. 
This is where most of the attention of the present paper is devoted to. 
We emphasise that we provide the corresponding arguments for smooth, convex as well as cube-like domains in an accessible manner. 

Particularly, due to the very elementary form of the Friedrichs' estimate, 
the closed range results themselves are shown for (unbounded) cuboids only, first. 
Closed range results for domains with curved boundaries are then established 
by translating closed range result from cuboids to other domains using bi-Lipschitz transformations. 
This then also includes convex domains.

The upshot of it all can be summarised rather neatly, which is nurtured from its correctness for cuboids 
and which will be shown in the course of this manuscript in our main Theorem \ref{theo:main}. 
For $\om$ being the image of a global bi-Lipschitz transformation of a cuboid the following equivalences are true:
\begin{align*}
&\bullet
&
R(\nac_{\om})\text{ closed} 
&\quad\;\equi
&
R(\div_{\om})\text{ closed} 
&\qqequi
\om\text{ is bounded in one direction.}\\
&\bullet
&
R(\rotc_{\om})\text{ closed} 
&\quad\;\equi
&
R(\rot_{\om})\text{ closed} 
&\qqequi
\om\text{ is bounded in two directions.}\\
&\bullet
&
R(\divc_{\om})\text{ closed} 
&\quad\;\equi
&
\hspace*{-2ex}
R(\na_{\om})\text{ closed} 
&\qqequi
\om\text{ is bounded in three directions}.
\end{align*}
Surprisingly, the example with mixed boundary conditions presented at the end of Section \ref{sec:clrancubes} 
shows that $\rot$ can have a closed range even if $\om$ is \emph{bounded in just one direction}.

Coming back to the Maxwell operator from the beginning, and, in fact, 
employing examples induced from the wave equation, 
we emphasise that our results are generally applicable for all kinds of wave propagation 
in wave guides $\om$ being bounded in one or two directions,
which appear to be of high interest.

Before we describe the course of the manuscript, 
we revisit parts of the literature mainly concerned with Gaffney's inequality. 
Generally, Gaffney's estimate (including the regularity part) is well-known, 
and there are standard references that provide the respective content, see, among others, 
\cite{S1982a,G1985a,C1991a} or \cite[Theorem 2.17]{ABDG1998a},
\cite[Lemma 3.2, Appendix A]{P2019a}. 
Nevertheless, since there are, to the best of our knowledge, 
no sources for easy digestion, 
we decided to provide self-contained and simple proofs using rather elementary computations 
without any serious technical difficulties. 
In any case, it might be interesting to know that the compactness 
of the Maxwell embedding for general bounded weak Lipschitz domains $\om$,
cf.~\cite{P1984a},
makes use only of Gaffney's inequality for the unit cube or the unit ball,
cf.~Theorem \ref{theo:gaffneyconvexbd} and Theorem \ref{theo:gaffneyconvex},
the transformation theorem, cf.~Theorem \ref{theo:transtheo},
and the classical Rellich--Kondrachov selection theorem for $\H^{1}(\om)$ functions. 
This result is often referred to as the
Picard--Weber--Weck selection theorem for bounded weak Lipschitz domains,
cf.~\cite{W1974a} and \cite{W1980a,C1990a,W1993a,J1997a,BPS2016a}. Moreover, Gaffney's estimate (and the included regularity) is also the main tool for proving that the Maxwell constant is bounded from above
by the Poincar\'e constant if the underlying domain is convex,
cf.~\cite{P2015b,P2017a,P2019a}.
In other words, the first positive Maxwell eigenvalue is bounded from below
by the square root of the first positive Neumann eigenvalue of the Laplacian.

We briefly sketch the course of this manuscript. 
After having provided a list of substantial implications of a closed range for (abstract) 
operators in Section \ref{sec:app}, in Section \ref{sec:ibp}, 
we give simple proofs of integration by parts formulas for smooth domains as well as cuboids. 
Then, in Section \ref{sec:gaffneyineq}, 
we establish that convex domains (independently of any boundedness) are Gaffney domains. 
We also show that cuboids are exact Gaffney domains. 
In passing, we show that the space of harmonic Dirichlet or Neumann vector fields is trivial for convex domains. 
Section \ref{sec:clrancubes} is devoted to characterise closed range results 
for (possibly unbounded) cuboids depending on the number of their directions of boundedness. 
This entails the application of the above mentioned strategy and, at the same time, 
it involves counterexamples showing that the directions of boundedness estimates yielding closed range are sharp. 
This is complemented by the above mentioned example of a realisation 
of $\rot$ with mixed boundary conditions and closed range on a domain being boundedness in one direction, only.
In Section \ref{sec:clranlip} we generalise our results to admissible global Lipschitz domains 
and provide some explicit examples. 
The technical background needed for this section is provided in Appendix \ref{sec:transtheo}, 
where we establish the \emph{transformation theorem}, 
for which we mention \cite{zbMATH01868846} as a different source for the particular case of $\C^{1,1}$-transformations.

Next, we turn to applications of closed range results. 
Reader familiar with the functional analytic consequences of operators 
with closed range rather interested in the actual proof of the closed range statements 
along with Gaffney's inequality in the situations mentioned may skip the next section entirely. 
However, note that the consequences of a closed range are remarkably profound 
and shed functional analytic light on problems frequently addressed 
in applied PDEs like low frequency asymptotics or exponential stability.

\section{Applications}
\label{sec:app}

The main property of ((skew)-selfadjoint) operators with closed range is a spectral gap around the origin. 
In fact, this is the core result that is being proved 
and used in all the of the following more concrete applications to follow. 
This spectral gap is based on the following extension of Theorem \ref{theo:fatblem1}, 
where we assume as it is done throughout this section that 
$$\A:D(\A)\subset\H_{0}\to\H_{1}$$
is a closed and densely defined linear operator from Hilbert space $\H_{0}$ to Hilbert space $\H_{1}$. 
We introduce its restriction to $N(\A)^{\bot}=N(\A)^{\bot_{\H_{0}}}=\ol{R(\A^{*})}$, 
the corresponding \emph{reduced version}, given by
$$\wh{\A}\coloneqq\A|_{N(\A)^{\bot}}:D(\wh{\A})\subset N(\A)^{\bot}\to\ol{R(\A)},
\qquad
D(\wh{\A})\coloneqq D(\A)\cap N(\A)^{\bot}.$$

\begin{theo}
\label{theo:bddinv} 
The following conditions are equivalent:
\begin{itemize}
\item[\bf(i)] 
$R(\A)\subset\H_{1}$ is closed.
\hfill{\rm\bf(closed range)}
\item[\bf(iii)] 
$\wh{\A}^{-1}:R(\A)\to D(\wh{\A})$ is bounded.
\hfill{\rm\bf(bounded inverse)}
\end{itemize}
\end{theo}

The well-known closed range theorem furthermore asserts that $\A$ and $\A^{*}$ have closed range only simultaneously. 
As a consequence, the corresponding reduced operators are simultaneously continuously invertible.

\begin{theo}[Banach's closed range theorem]
\label{theo:bcrt}
$$R(\A)\text{ closed in }\H_{1}
\qequi R(\A^{*})\text{ closed in }\H_{0}$$
\end{theo}

There are plenty of applications for operators with closed range.
Amongst these, the so-called FA-ToolBox,
cf.~\cite{P2019b,P2020a,PS2022a,PZ2019a,PZ2023a}, 
which provides techniques for solving linear equations in the context of closed Hilbert complexes. 
The mentioned references also contain proofs of Theorems \ref{theo:bddinv} and \ref{theo:bcrt}. 
In the following lines, we rather focus on consequences of a combination of these two theorems. 
In fact, there are three operators obtained by standard constructions 
from $\A$ that also have closed range and to which the observations to come particularly apply to. 
These are
$$\T_{1}\coloneqq\A^{*}\A,
\qquad
\T_{2}\coloneqq\begin{bmatrix}0&\A^{*}\\\A&0\end{bmatrix},
\qquad
\T_{3}\coloneqq\begin{bmatrix}0&-\A^{*}\\\A&0\end{bmatrix},$$
defined as operator on their natural domains in $\H_{0}$, $\H_{0}\times\H_{1}$, and $\H_{0}\times\H_{1}$, respectively. 
Note that the former two operators are self-adjoint, whereas the latter is skew-selfadjoint. 
Indeed, the results being elementary calculations for $\T_{2}$ and $\T_{3}$. 
For $\T_{1}$, the corresponding result can be deduced by considering $(1+\T_{3})^{-1}(1-\T_{3})^{-1}$, 
which is well-defined by the skew-selfadjointness of $\T_{3}$ 
and self-adjoint itself\footnote{This idea of proof 
for the self-adjointness of $T_{1}$ has been communicated 
to us by Rainer Picard. 
The detailed argument can be found in \cite[Proposition B.4.17]{zbMATH07214200}.}. 
We recall the following standard set-up in the context of the de Rham complex 
in order to have a rich example class in mind for the abstract results to follow.

\begin{ex}
We recall the de Rham complex \eqref{eq:derham1}, i.e.,
\begin{equation}
\label{eq:derham2}
\def\arrowlength{12ex}
\def\arrowdistance{.8}
\begin{tikzcd}[column sep=\arrowlength]
\Lt(\om) 
\ar[r, rightarrow, shift left=\arrowdistance, "\A_{0}=\nac"] 
\ar[r, leftarrow, shift right=\arrowdistance, "\A_{0}^{*}=-\div"']
&
\Lt(\om) 
\ar[r, rightarrow, shift left=\arrowdistance, "\A_{1}=\rotc"] 
\ar[r, leftarrow, shift right=\arrowdistance, "\A_{1}^{*}=\rot"']
& 
\Lt(\om) 
\arrow[r, rightarrow, shift left=\arrowdistance, "\A_{2}=\divc"] 
\arrow[r, leftarrow, shift right=\arrowdistance, "\A_{2}^{*}=-\na"']
& 
\Lt(\om),
\end{tikzcd}
\end{equation}
and introduce the negative Dirichlet and Neumann Laplacians
\begin{align*}
-\Delta_{\dir}&\coloneqq\A_{0}^{*}\A_{0}=-\div\nac,
&
-\Delta_{\neu}&\coloneqq\A_{2}\A_{2}^{*}=-\divc\na,
\intertext{the negative Dirichlet and Neumann $\na$-$\div$ operators}
-\lozenge_{\dir}&\coloneqq\A_{2}^{*}\A_{2}=-\na\divc,
&
-\lozenge_{\neu}&\coloneqq\A_{0}\A_{0}^{*}v=-\nac\div,
\intertext{the negative Dirichlet and Neumann double-$\rot$ operators}
\square_{\dir}&\coloneqq\A_{1}^{*}\A_{1}=\rot\rotc,
&
\square_{\neu}&\coloneqq\A_{1}\A_{1}^{*}=\rotc\rot,
\end{align*}
and the negative vector Laplacians
\begin{align*}
-\vec{\Delta}_{\dir}&\coloneqq\A_{1}^{*}\A_{1}+\A_{0}\A_{0}^{*}=\square_{\dir}-\lozenge_{\neu}=\rot\rotc-\nac\div,\\
-\vec{\Delta}_{\neu}&\coloneqq\A_{1}\A_{1}^{*}+\A_{2}^{*}\A_{2}=\square_{\neu}-\lozenge_{\dir}=\rotc\rot-\na\divc.
\end{align*}
respectively. All of them are selfadjoint.
Moreover, we shall discuss the skew-selfadjoint operators 
\begin{align*}
\S_{0}&\coloneqq\begin{bmatrix}0&-\A_{0}^{*}\\\A_{0}&0\end{bmatrix}
=\begin{bmatrix}0&\div\\\nac&0\end{bmatrix},
&
\S_{1}&\coloneqq\begin{bmatrix}0&-\A_{1}^{*}\\\A_{1}&0\end{bmatrix}
=\begin{bmatrix}0&-\rot\\\rotc&0\end{bmatrix},\\
\S_{2}&\coloneqq\begin{bmatrix}0&-\A_{2}^{*}\\\A_{2}&0\end{bmatrix}
=\begin{bmatrix}0&\na\\\divc&0\end{bmatrix},
\end{align*}
occurring, among others, in linear acoustics and Maxwell's equations.
\end{ex}

Next, we turn to consequences of the closed range property 
for block operators of the form mentioned in the concluding lines of the previous examples.

\subsection{(Skew-)Selfadjoint Operators with Closed Range}

Throughout, let $\H$ be a Hilbert space and
$$\T\colon D(\T)\subset\H\to\H$$
be skew-selfadjoint or self-adjoint. 
Recall the notation for the reduced operator $\wh{\T}$, 
which turns out to be continuously invertible as long as $R(\T)\subset\H$ is closed. 
In other words, we have $0\in\rho(\wh{\T})$. 
By the openness of the reslvent set $\rho(\wh{\T})$, 
we obtain that $\wh{\T}-\lambda$ is continuously invertible in a neighbourhood of $0$. The quantified statement reads as follows.

\begin{lem}[spectral gap for the reduced operator]
\label{lem:fatblem3}
Let $R(\T)$ be closed and let $|\lambda|<1/c_{\T}$
with $c_{\T}\coloneqq\bnorm{\wh{\T}^{-1}}_{R(\T)\to R(\T)}$.
Then 
$$\forall\,x\in D(\wh{\T})\qquad
\norm{x}_{\H}\leq\wh{c}_{\T,\lambda}\bnorm{(\T-\lambda)x}_{\H},\qquad\quad
\wh{c}_{\T,\lambda}\coloneqq\frac{c_{\T}}{1-c_{\T}|\lambda|},$$
and 
$$N(\wh{\T}-\lambda)=\{0\},
\qquad
R(\wh{\T}-\lambda)=R(\T),$$
in particular, $R(\wh{\T}-\lambda)$ is closed.
Moreover, the inverse $(\wh{\T}-\lambda)^{-1}:R(\T)\to D(\wh{\T})$ is bounded with 
$\bnorm{(\wh{\T}-\lambda)^{-1}}_{R(\T)\to R(\T)}
\leq\wh{c}_{\T,\lambda}$.
In other words,
$B(0,1/c_{\T})\subset\rho(\wh{\T})$.
\end{lem}

\begin{proof}
As mentioned above,
$\wh{\T}^{-1}:R(\T)\to D(\wh{\T})$ is bounded, and
$$\norm{x}_{\H}
\leq c_{\T}\norm{\T x}_{\H}
\leq c_{\T}\bnorm{(\T-\lambda)x}_{\H}+c_{\T}|\lambda|\norm{x}_{\H}$$
holds for $x\in D(\wh{\T})$,
showing the estimate for all $|\lambda|<1/c_{\T}$. 
Hence
$N(\wh{\T}-\lambda)=\{0\}$
and $R(\wh{\T}-\lambda)$ is closed with 
$$R(\wh{\T}-\lambda)
=N(\wh{\T}^{*}-\overline{\lambda})^{\bot_{R(\T)}}
=N(\wh{\T}-\overline{\lambda})^{\bot_{R(\T)}}
=R(\T).$$
Thus $\wh{(\wh{\T}-\lambda)}=\wh{\T}-\lambda$.
Therefore, 
$$\wh{(\wh{\T}-\lambda)}^{-1}=(\wh{\T}-\lambda)^{-1}:
R(\wh{\wh{\T}-\lambda})=R(\wh{\T}-\lambda)=R(\T)
\to D(\wh{\wh{\T}-\lambda})=D(\wh{\T}-\lambda)=D(\wh{\T})$$
is bounded by the above estimate and Theorem \ref{theo:bddinv}.
\end{proof}

The spectral gap for the reduced operators has consequence also for the non-reduced operator.

\begin{theo}[spectral gap]
\label{theo:fatblem4}
Let $R(\T)\subset\H$ closed; $0<|\lambda|<1/c_{\T}$, where $c_{\T}\coloneqq\bnorm{\wh{\T}^{-1}}_{R(\T)\to R(\T)}$.
Then $N(\T-\lambda)=0$ and $R(\T-\lambda)=\H$. Moreover, 
$$(\T-\lambda)^{-1}:\H\to D(\T)$$
is bounded with 
$\bnorm{(\T-\lambda)^{-1}}_{\H\to\H}\leq c_{\T,\lambda}$, 
where 
$\displaystyle c_{\T,\lambda}
\coloneqq\sqrt{\wh{c}_{\T,\lambda}^{2}+|\lambda|^{-2}}.$
In particular,
$$\forall\,x\in D(\T)
\qquad
\norm{x}_{\H}\leq c_{\T,\lambda}\bnorm{(\T-\lambda)x}_{\H}.$$
In other words,
$B(0,1/c_{\T})\setminus\{0\}\subset\rho(\T)$.
\end{theo}

\begin{proof}
Let $x\in D(\T)$ with $(\T-\lambda)x=f\in\H.$ 
According to the standard orthogonal decomposition, $\H=R(\T)\oplus_{\H}N(\T)$, we infer $D(\T)=D(\wh{\T})\oplus_{\H}N(\T)$.
We see
\begin{align*}
D(\T)\ni x&=x_{R}+x_{N}\in D(\wh{\T})\oplus_{\H}N(\T),
&
D(\wh{\T})&=D(\T)\cap R(\T),\\
\H\ni f&=f_{R}+f_{N}\in R(\T)\oplus_{\H}N(\T),
\end{align*}
and obtain the equation
$(\T-\lambda)x_{R}-\lambda x_{N}=f_{R}+f_{N}$,
which separates into the two equations
$$(\wh{\T}-\lambda)x_{R}=f_{R}\in R(\T),
\qquad
-\lambda x_{N}=f_{N}\in N(\T)$$
by orthogonality. Lemma \ref{lem:fatblem3} yields
$$x_{R}=(\wh{\T}-\lambda)^{-1}f_{R},
\qquad
x_{N}=-\frac{1}{\lambda}f_{N},$$
and thus 
$$\norm{x}_{\H}^{2}
=\norm{x_{R}}_{\H}^{2}
+\norm{x_{N}}_{\H}^{2}
\leq\wh{c}_{\T,\lambda}^{2}\norm{f_{R}}_{\H}^{2}
+|\lambda|^{-2}\norm{f_{N}}_{\H}^{2}
\leq c_{\T,\lambda}^{2}\norm{f}_{\H}^{2}.$$
We conclude\footnote{Note that for $\lambda\neq0$ we always have $N(\T-\lambda)=N(\wh{\T}-\lambda)$.} 
$N(\T-\lambda)=\{0\}$, $\wh{\T-\lambda}=\T-\lambda$, and, by
Theorem \ref{theo:fatblem1}, $R(\T-\lambda)$ is closed and hence equals $\H$. As a consequence of Theorem \ref{theo:bddinv}, we get
$(\T-\lambda)^{-1}:\H\to D(\T)$
is bounded with 
$\bnorm{(\T-\lambda)^{-1}}_{\H\to\H}\leq c_{\T,\lambda}$.
We emphasise that indeed 
$x\coloneqq x_{R}+x_{N}=(\wh{\T}-\lambda)^{-1}f_{R}-\frac{1}{\lambda}f_{N}$
for $f\in\H$ solves $(\T-\lambda)x=f_{R}+f_{N}=f$.
\end{proof}

\begin{rem}
\label{rem:fatoolbox5}
The latter proof shows
$$(\T-\lambda)^{-1}=(\wh{\T}-\lambda)^{-1}\pi_{R(\T)}-\frac{1}{\lambda}\pi_{N(\T)}$$
with orthogonal projectors $\pi_{R(\T)}$ and $\pi_{N(\T)}$
onto the range and kernel of $\T$, respectively.
\end{rem}

The latter remark can be slightly extended to obtain the following statement 
about low frequency asymptotics, see, e.g., 
\cite{P2006a,P2008b,PO2020a,zbMATH04210630,zbMATH00147080,zbMATH00147438,zbMATH01089056,zbMATH03859617,zbMATH03985775} 
for results of this kind in the context of (acoustic, elastic, electro-magnetic) wave propagation phenomena.

\begin{theo}[low frequency asymptotics]
\label{theo:fatblem5}
Let $R(\T)$ be closed
and let $0<|\lambda|<1/c_{\T}$.
Then
\begin{align*}
(\wh{\T}-\lambda)^{-1}
&=\sum_{n=0}^{\infty}\lambda^{n}\wh{\T}^{-n-1},\\
(\T-\lambda)^{-1}
&=-\frac{1}{\lambda}\pi_{N(\T)}
+(\wh{\T}-\lambda)^{-1}\pi_{R(\T)}\\
&=-\frac{1}{\lambda}\pi_{N(\T)}
+\sum_{n=0}^{k-1}\lambda^{n}\wh{\T}^{-n-1}\pi_{R(\T)}
+\lambda^{k}\wh{\T}^{-k-1}\sum_{n=0}^{\infty}\lambda^{n}\wh{\T}^{-n}\pi_{R(\T)}
\end{align*}
and 
$$\bnorm{(\T-\lambda)^{-1}
+\frac{1}{\lambda}\pi_{N(\T)}
-\sum_{n=0}^{k-1}\lambda^{n}\wh{\T}^{-n-1}\pi_{R(\T)}}_{\H\to\H}
\leq\wh{c}_{\T,\lambda}c_{\T}^{k}|\lambda|^{k}
=\calO(\lambda^{k})\quad(\text{for }\lambda\to0).$$
\end{theo}

\begin{proof}
We observe
$(\wh{\T}-\lambda)=\wh{\T}\big(1-\lambda\wh{\T}^{-1}\big)$
and $\norm{\lambda\wh{\T}^{-1}}_{R(\T)\to R(\T)}=|\lambda|c_{\T}<1$.
Thus by Neumann's series
$$(\wh{\T}-\lambda)^{-1}
=\big(1-\lambda\wh{\T}^{-1}\big)^{-1}\wh{\T}^{-1}
=\sum_{n=0}^{\infty}\lambda^{n}\wh{\T}^{-n-1}
=\sum_{n=0}^{k-1}\lambda^{n}\wh{\T}^{-n-1}
+\lambda^{k}\wh{\T}^{-k-1}\sum_{n=0}^{\infty}\lambda^{n}\wh{\T}^{-n},$$
which, together with Remark \ref{rem:fatoolbox5}, shows the equations.
Moreover, 
$$\bnorm{\wh{\T}^{-k-1}\sum_{n=0}^{\infty}\lambda^{n}\wh{\T}^{-n}\pi_{R(\T)}}_{\H\to\H}
\leq c_{\T}^{k+1}\bnorm{\sum_{n=0}^{\infty}\lambda^{n}\wh{\T}^{-n}}_{R(\T)\to R(\T)}
\leq\frac{c_{\T}^{k+1}}{1-|\lambda|c_{\T}}
=\wh{c}_{\T,\lambda}c_{\T}^{k},$$
concludes the proof.
\end{proof}

\subsection{Exponential Stability}

We quickly introduce the operator-theoretic setting for space-time equations 
in the context of evolutionary equations introduced by Picard, \cite{zbMATH05599240}. 
We also refer to the monograph \cite{STW22} accessible for graduate students, and, particularly, 
we refer to \cite[Chapter 11]{STW22} on exponential stability. 
The core result to deduce exponential stability for (nonlinear, time-nonlocal) Maxwell's equations in \cite{DTW24} 
is an exponential stability statement for evolutionary equations that can be found in \cite[Corollary 11.6.1]{STW22}. 
In the following, we sketch the set-up and provide a corresponding stability result for Maxwell type equations. 
For this, we introduce for $\nu\in\reals$, the Hilbert space
$$\Lt_{\nu}(\reals;\H)\coloneqq
\Big\{f\in\Lt_{\textnormal{loc}}(\reals;\H):\int_{\reals}\bnorm{f(t)}_{\H}^{2}\exp(-2\nu t)\,dt<\infty\Big\},$$
endowed with the obvious scalar product. 
Introducing the (distributional, time-) derivative $\p_{t,\nu}$ 
on this space with maximal domain, 
the following result is a standard application of Picard's well-posedness theorem, 
where we understand that operators defined on $\H$ can be lifted canonically 
to operators on $\Lt_{\nu}(\reals;\H)$ retaining their properties 
(e.g., self-adjointness, positivity, etc.).
For ease of readability, we will re-use the notation of the original operator also for the lifted one.

\begin{theo}[{{Picard's theorem, \cite[Theorem 6.2.1]{STW22}}}]
\label{theo:pt} 
Let $0\leq\M_{0}=\M_{0}^{*},\M_{1}\in\L(\H)$, $\S$ 
be a skew-selfadjoint operator in $\H$. 
Then, if there exists $\nu_{0}\geq 0$ such that 
$$\forall\phi\in\H\qquad\nu_{0}\langle\M_{0}\phi,\phi\rangle_{\H} +\Re\langle\M_{1}\phi,\phi\rangle_{\H}\geq c\|\phi\|_{\H},$$
the operator
$$\wt{\B}_{\nu}\coloneqq (\p_{t,\nu}\M_{0}+\M_{1}+\S)$$
is closable in $\Lt_{\nu}(\reals;\H)$ for all $\nu\geq\nu_{0}$, the closure of which, 
$\B_{\nu}\coloneqq\overline{\wt{\B}_{\nu}}$, is continuously invertible. 
Moreover, for $\nu,\mu\geq\nu_{0}$ and $f\in \Lt_{\nu}(\reals;\H)\cap\Lt_\mu(\reals;\H)$, 
we have $\B_{\nu}^{-1}f =\B_{\mu}^{-1}f$.
\end{theo}

A particular application can be found in the following example of Maxwell type.

\begin{ex}[Maxwell type equations]
\label{ex:Mt} 
Let $\A\colon D(\A)\subset\H_{0}\to\H_{1}$ be closed and densely defined, 
$0<c\leq\varepsilon =\varepsilon^{*},\sigma\in\L(\H_{0})$, 
$0<c\leq\mu=\mu^{*}\in\L(\H_{1})$ in the sense of positive definiteness. 
Then
$$(\p_{t,\nu}\M_{0}+\M_{1}+\S)\coloneqq
\Big(
\p_{t,\nu}
\begin{bmatrix}
\varepsilon&0\\
0&\mu
\end{bmatrix}
+
\begin{bmatrix}
\sigma&0\\
0&0
\end{bmatrix}
+
\begin{bmatrix}
0&-\A^{*}\\
\A&0
\end{bmatrix}
\Big)$$
satisfies the assumptions of Theorem \ref{theo:pt}. 
Note the particular choice $\A=\rotc_{\om}$ for some open $\om\subset\reals ^3$.
\end{ex}

It has been found that evolutionary equations lead to a convenient framework 
to express results concerning exponential stability. 
In fact, an operator of the form of $\B_{\nu}$ as provided in Theorem \ref{theo:pt} is \textbf{exponentially stable}, 
if there exists $\eta>0$ such that for all 
$\nu\geq\nu_{0}$ 
and  $f\in \Lt_{\nu}(\reals;\H)\cap\Lt_{-\eta}(\reals;\H)$ we have
$$\B_{\nu} f\in\Lt_{-\eta}(\reals;\H).$$
The relationship of this notion of exponential stability to more classical formulations 
is expressed in \cite[Chapter 11]{STW22}. 
Next, as before, closed range results can help establish exponential stability statements. 
For simplicity, we focus on equations of the form provided in Example \ref{ex:Mt}.

\begin{theo}
\label{theo:expstMt} 
Let $\A\colon D(\A)\subset\H_{0}\to\H_{1}$ 
be closed and densely defined with closed range, and let $\varepsilon,\sigma,\mu>0$. 
Then, there exists $\eta>0$ such that for all $\nu>0$ 
and $f\in\Lt_{\nu}(\reals;\H_{0})\cap\Lt_{-\eta}(\reals;\H_{0})$ and 
$$U\coloneqq
\overline{\Big(\p_{t,\nu}
\begin{bmatrix}
\varepsilon&0\\
0&\mu
\end{bmatrix}
+
\begin{bmatrix}
\sigma &0\\
0&0
\end{bmatrix}
+
\begin{bmatrix}
0&-\A^{*}\\
\A&0
\end{bmatrix}
\Big)}^{-1}
\begin{bmatrix}
f\\
0
\end{bmatrix}$$
we have
$$U\in \Lt_{\nu}(\reals;\H)\cap\Lt_{-\eta}(\reals;\H).$$
\end{theo}

\begin{proof}
Before we proceed proving the actual result, 
we may assume without loss of generality, 
that $\varepsilon=1$ and $\mu=1$. 
Indeed, multiplying 
$$\overline{\Big(\p_{t,\nu}
\begin{bmatrix}
\varepsilon&0\\
0&\mu
\end{bmatrix}
+
\begin{bmatrix}
\sigma&0\\
0&0
\end{bmatrix}
+
\begin{bmatrix}
0&-\A^{*}\\
\A&0
\end{bmatrix}
\Big)}$$
from the left and the right with 
$\begin{bmatrix}
\varepsilon^{-1/2}&0\\
0&{\mu}^{-1/2}
\end{bmatrix}$ 
we obtain
$$\overline{\Big(\p_{t,\nu}
\begin{bmatrix}
1&0\\
0&1
\end{bmatrix}
+
\begin{bmatrix}
\wt{\sigma}&0\\
0&0
\end{bmatrix}
+
\begin{bmatrix}
0&-{\wt{\A}}^{*}\\
\wt{\A}&0
\end{bmatrix}\Big)},$$
where $\wt{\sigma}=\varepsilon^{-1/2}\sigma\varepsilon^{-1/2}$
and $\wt{\A}=\mu^{-1/2}\A\varepsilon^{-1/2}$.
Note that $\A$ has closed range, if and only if $\wt{\A}$ has; 
moreover $\wt{\sigma}=\sigma/\varepsilon>0$. 
Henceforth, we drop \,$\wt{\cdot}$\, again in our notation. 
Considering the abstract Helmholtz decomposition into 
$$\H_{0}\times\H_{1}
=R(\S)\oplus N(\S)
=\big(R(\A^{*})\times R(\A)\big)\oplus\big(N(\A)\times N(\A^{*})\big)$$
with 
$\S=
\begin{bmatrix}
0&-\A^{*}\\
\A&0
\end{bmatrix}$, 
we may rewrite the operator in question (similarly to the case for the low-frequency asymptotics) as
$$\overline{\Big(\p_{t,\nu}
\begin{bmatrix}
1&0\\
0&1
\end{bmatrix}
+\Sigma+
\begin{bmatrix}
\wh{S}&0\\
0&0
\end{bmatrix}\Big)},$$
where
$$\Sigma=
\begin{bmatrix}
\Sigma_{11}&0\\
0&\Sigma_{22}
\end{bmatrix},
\qquad
\Sigma_{11}=
\begin{bmatrix}
\pi_{R(\A^{*})}\sigma\pi_{R(\A^{*})}&0\\
0&0
\end{bmatrix},
\qquad
\Sigma_{22}=
\begin{bmatrix}
\pi_{N(\A)}\sigma\pi_{N(\A)}&0\\
0&0
\end{bmatrix}.$$
Hence, the equation in question decouples into the following two equations
$$\overline{(\p_{t,\nu}+\Sigma_{11}+\hat{S})}
\begin{bmatrix}
U_{0,R}\\
U_{1,R}
\end{bmatrix}
=
\begin{bmatrix}
f_R\\
0
\end{bmatrix},
\qquad
\overline{(\p_{t,\nu}+\Sigma_{22})}
\begin{bmatrix}
U_{0,N}\\
U_{1,N}
\end{bmatrix}
=
\begin{bmatrix}
f_N\\
0
\end{bmatrix},$$
where we use 
\begin{align*}
f=f_R+f_N&\in\Lt_{\nu}\big(\reals; R(\A^{*})\big)+\Lt_{\nu}\big(\reals; N(\A)),\\
U=U_{0}+U_{1}&\in\Lt_{\nu}(\reals;\H_{0})+\Lt_{\nu}(\reals;\H_{1}),\\
U_{0}=U_{0,R}+U_{0,N}&\in\Lt_{\nu}\big(\reals; R(\A^{*})\big)+\Lt_{\nu}\big(\reals; N(\A)\big),\\
U_{1}=U_{1,R}+U_{1,N}&\in\Lt_{\nu}\big(\reals; R(\A)\big)+\Lt_{\nu}\big(\reals; N(\A^{*})\big).
\end{align*}
The first equation is exactly of the form treated in \cite[Section 11.4]{STW22} and, thus, 
\cite[Corollary 11.6.1]{STW22} yields the claim for $(U_{0,R}$, $U_{1,R})$. 
Uniqueness of the solution leads to $U_{1,N}=0$ and $U_{0,N}$ 
satisfies the claimed asymptotics using the variation of constants formula for ODEs.
\end{proof}

With little more effort, 
the previous results can also be transferred to variable coefficients $\varepsilon,\mu,\sigma$. 
We refrain from following this path. 
Instead, we now turn to the announced rationale proving Gaffney's inequality in standard and, due to unboundedness, 
in less standard situations. 
The first line of questions revolves around integration by parts on smooth domains and cuboids.

\section{Integration by Parts}
\label{sec:ibp}

Let $\om\subset\rt$ be a bounded Lipschitz domain
with boundary $\ga=\p\om$ and outer unit normal field $\nu$.
For $E\in\Ci(\rt)$ we define its normal component on $\ga$ (a.e.) by 
$$E_{\nor}\coloneqq\nu\cdot E,$$
and its corresponding tangential component on $\ga$ (a.e.) by 
$$E_{\tan}\coloneqq E-E_{\nor}\nu=\nu\times E\times\nu.$$
For simplicity, here, we consider only real-valued vector fields.

For $E,F\in\Ci(\rt)$ with matrix fields $\na E,\na F\in\Ci(\rt)$ we have \emph{point-wise}
$$2\rot E\cdot\rot F
=\big(\na E-(\na E)^{\top}\big)\cdot\big(\na F-(\na F)^{\top}\big)
=2\na E\cdot\na F
-2\na E\cdot(\na F)^{\top},$$
where the dots stand for the standard scalar product for vectors 
and the Forbenius scalar product for matrices, respectively. 
Hence we observe the (point-wise) key relation
$$\rot E\cdot\rot F
+\div E\cdot\div F
=\na E\cdot\na F
-\sum_{k,l=1}^{3}
\Big(\p_{l}E_{k}\cdot\p_{k}F_{l}-\p_{k}E_{k}\cdot\p_{l}F_{l}\Big).$$

Integration over $\om$, Gau{\ss}' theorem, and Schwarz' lemma yield
\begin{align}
\begin{aligned}
\label{eq:ibp1}
&\qquad
\int_{\om}\na E\cdot\na F
-\int_{\om}\rot E\cdot\rot F
-\int_{\om}\div E\cdot\div F\\
&=\sum_{k,l=1}^{3}
\int_{\om}\Big(\p_{l}E_{k}\p_{k}F_{l}-\p_{k}E_{k}\p_{l}F_{l}\Big)\\
&=\sum_{k,l=1}^{3}
\int_{\ga}\Big(\nu_{l}E_{k}\p_{k}F_{l}-\nu_{k}E_{k}\p_{l}F_{l}\Big)
=\int_{\ga}\Big((\p_{E}F)_{\nor}-E_{\nor}\div F\Big)
=:\cI_{\ga}(E,F).
\end{aligned}
\end{align}

So, it remains to investigate the last two boundary integrals,
which we shall do in the following two subsections for smooth domains $\om$
and for the special case of the unit cube $\om=Q=(0,1)^{3}$, separately.

\subsection{Smooth Domains}
\label{sec:ibpsd}

Let $\om\subset\rt$ be a bounded and smooth domain,
e.g., $\om$ is of class $\Ci$.
Then the unit normal field can be extended into a neighbourhood of $\ga$
such that the resulting vector field (still denoted by) $\nu$ satisfies in this neighbourhood
$|\nu|^{2}=1$, $0=\p_{k}|\nu|^{2}=2\nu\cdot\p_{k}\nu$, and $\rot\nu=0$, i.e.,
\begin{align}
\label{eq:normal}
|\nu|=1,
\qquad
(\na\nu)\nu=0,
\qquad
\rot\nu=0,
\end{align}
cf.~\cite{G1985a,ABDG1998a,C1991a}. 
In fact, this property can be achieved by studying the signed distance function and compute the gradient of which. 
Then locally around the boundary, the unit outward normal becomes a gradient, 
particularly satisfying the last requirement, see \cite{zbMATH08072166} for a recent reference.
Note that $\na\nu$ is the (symmetric) second fundamental form
and that $-\frac{1}{2}\div\nu$ is the mean curvature of $\ga$.
Hence, both, $\na\nu$ and $\div\nu$, are non-negative for convex domains $\om$.

We modify $\cI_{\ga}(E,F)$ in \eqref{eq:ibp1}, still for $E,F\in\Ci(\rt)$, by
\begin{align}
\begin{aligned}
\label{eq:ibpIga}
\cI_{\ga}(E,F)
&\;=\sum_{k,l=1}^{3}
\int_{\ga}\Big(
E_{k}\p_{k}(\nu_{l}F_{l})
-E_{k}(\p_{k}\nu_{l})F_{l}
-\nu_{k}E_{k}\p_{l}F_{l}
\Big)\\
&=\int_{\ga}\Big(
E\cdot\na F_{\nor}
-E\cdot((\na\nu) F)
-E_{\nor}\div F
\Big)\\
&=\int_{\ga}\Big(
E_{\tan}\cdot\na_{\tan}F_{\nor}
-E_{\nor}\div F_{\tan}
-E_{\nor}F_{\nor}\div\nu
-E_{\tan}\cdot(\na\nu F_{\tan})
\Big),
\end{aligned}
\end{align}
as we have, recalling that $E_{\nor}$ is a scalar and $E_{\tan}$ a vector field introduced at the beginning of this section,
\begin{align*}
E\cdot\na F_{\nor}
&=E_{\tan}\cdot\na F_{\nor}
+E_{\nor}\,\nu\cdot\na F_{\nor}
=E_{\tan}\cdot\na_{\tan}F_{\nor}
+E_{\nor}\,\nu\cdot\na F_{\nor},\\
E_{\nor}\div F
&=E_{\nor}\div F_{\tan}
+E_{\nor}F_{\nor}\div\nu
+E_{\nor}\,\nu\cdot\na F_{\nor},\\
E\cdot(\na\nu F)
&=E_{\tan}\cdot(\na\nu\,F_{\tan})
+F_{\nor}E_{\tan}\cdot(\na\nu\,\nu)
+E_{\nor}\nu\cdot(\na\nu\,F_{\tan})
+E_{\nor}F_{\nor}\nu\cdot(\na\nu\,\nu)
=E_{\tan}\cdot(\na\nu\,F_{\tan}),
\end{align*}
since $\na\nu\,\nu=0$.
Here, the surface gradient $\na_{\tan}$ is defined by $\na_{\tan}u\coloneqq (\na u)_{\tan}$.
Moreover, with $\rot\nu=0$ we see
\begin{align}
\label{eq:surfdiv}
\div F_{\tan}
=-\nu\cdot\rot(F\times\nu).
\end{align}
Let $\varphi\in\Ci(\rt)$ be supported in a small neighbourhood of $\ga$
with $\varphi=1$ in an even smaller neighbourhood. 
Then we obtain by \eqref{eq:surfdiv} and Gau{\ss}' theorem
\begin{align}
\begin{aligned}
\label{eq:ibp2}
\int_{\om}\na(\varphi E_{\nor})\cdot\rot(\varphi F\times\nu)
&=\int_{\om}\div\big(\varphi E_{\nor}\rot(\varphi F\times\nu)\big)\\
&=\int_{\ga}\nu\cdot\big(\varphi E_{\nor}\rot(\varphi F\times\nu)\big)
=-\int_{\ga}E_{\nor}\div F_{\tan},\\
\int_{\om}\rot(\varphi E\times\nu)\cdot\na(\varphi F_{\nor})
&=\int_{\om}\div\big((\varphi E\times\nu)\times\na(\varphi F_{\nor})\big)\\
&=\int_{\ga}\nu\cdot\big((\varphi E\times\nu)\times\na(\varphi F_{\nor})\big)
=\int_{\ga}\na F_{\nor}\cdot E_{\tan}
=\int_{\ga}E_{\tan}\cdot\na_{\tan}F_{\nor}.
\end{aligned}
\end{align}
Finally, we plug \eqref{eq:ibp2} into \eqref{eq:ibpIga}
and arrive at:

\begin{lem}[integration by parts for smooth domains]
\label{lem:ibpsmooth}
Let $\om$ be bounded and smooth
and let $E,F\in\H^{1}(\om)$. 
Then
$$\scp{\na E}{\na F}_{\Lt(\om)}
=\scp{\rot E}{\rot F}_{\Lt(\om)}
+\scp{\div E}{\div F}_{\Lt(\om)}
+\cI_{\ga}(E,F),$$
where $\cI_{\ga}$ is a boundary integral term given by
\begin{align*}
\cI_{\ga}(E,F)
&=\bscp{\rot(\varphi E\times\nu)}{\na(\varphi F_{\nor})}_{\Lt(\om)}
+\bscp{\na(\varphi E_{\nor})}{\rot(\varphi F\times\nu)}_{\Lt(\om)}\\
&\qquad\qquad
-\bscp{E_{\nor}}{(\div\nu)F_{\nor}}_{\Lt(\ga)}
-\bscp{E_{\tan}}{(\na\nu)F_{\tan}}_{\Lt(\ga)}.
\end{align*}
In particular, 
\begin{align*}
\norm{\na E}_{\Lt(\om)}^{2}
&=\norm{\rot E}_{\Lt(\om)}^{2}
+\norm{\div E}_{\Lt(\om)}^{2}
+\cI_{\ga}(E,E),\\
\cI_{\ga}(E,E)
&=2\bscp{\na(\varphi E_{\nor})}{\rot(\varphi E\times\nu)}_{\Lt(\om)}
-\bscp{E_{\nor}}{(\div\nu)E_{\nor}}_{\Lt(\ga)}
-\bscp{E_{\tan}}{(\na\nu)E_{\tan}}_{\Lt(\ga)}.
\end{align*}
\end{lem}

\begin{proof}
For $E,F\in\Ci(\rt)$ the assertions follow by the previous considerations and computations.
By approximation, i.e., $\ol{\Ci(\rt)\cap\H^{1}(\om)}^{\H^{1}(\om)}=\H^{1}(\om)$,
the results carry over to $E,F\in\H^{1}(\om)$. 
For this note that, as $\om$ is smooth, 
the mapping $\H^{1}(\om)\ni E\mapsto E|_{\ga}\in\Lt(\ga)$ 
is well-defined and continuous.
\end{proof}

\begin{rem}[integration by parts on the boundary]
\label{rem:ibpsmooth}
For $E,F\in\Ci(\om)$ we have by \eqref{eq:ibp2}
the following integration by parts formula on the boundary
$$\scp{\na_{\tan}E_{\nor}}{F_{\tan}}_{\Lt(\ga)}
=\bscp{\na(\varphi E_{\nor})}{\rot(\varphi F\times\nu)}_{\Lt(\om)}
=-\scp{E_{\nor}}{\div F_{\tan}}_{\Lt(\ga)}.$$
For $E,F\in\H^{1}(\om)$ this formula remains valid in the sense of traces for
the respective Sobolev spaces $\H^{1}(\om)$, $\H(\div,\om)$, and $\H(\rot,\om)$.
More precisely, we see by the complex properties
that $\na(\varphi E_{\nor}),\,\varphi F\times\nu\in\H(\rot,\om)$
as well as 
$\varphi E_{\nor}\in\H^{1}(\om)$ and $\rot(\varphi F\times\nu)\in\H(\div,\om)$.
Hence 
$$\bdual{\trt\na(\varphi E_{\nor})}{\trtx(\varphi F\times\nu)}_{\ga}
=\bscp{\na(\varphi E_{\nor})}{\rot(\varphi F\times\nu)}_{\Lt(\om)}
=\bdual{\trs(\varphi E_{\nor})}{\trn\rot(\varphi F\times\nu)}_{\ga},$$
where $\dual{\cdot}{\cdot}_{\ga}$ denotes (roughly) 
the duality in the respective $\H^{\pm1/2}(\ga)$ trace spaces 
without going into details.
Here, $\trs$, $\trn$, $\trt$, $\trtx$
denote the scalar, normal, and tangential, twisted tangential traces, respectively.
We emphasise that modifying (or even identifying) the terms 
$\div F_{\tan}$ and $\trn\rot(\varphi F\times\nu)$
to the proper surface divergence $\div_{\tan}$
requires some additional efforts, 
which are not relevant for our needs.
\end{rem}

\begin{cor}[integration by parts for smooth domains and homogeneous boundary conditions]
\label{cor:ibpsmooth}
Let $\om$ be bounded and smooth
and let $E,F\in\H^{1}(\om)$. 
If $E,F\in\Hc(\rot,\om)$, then
\begin{align*}
\scp{\na E}{\na F}_{\Lt(\om)}
&=\scp{\rot E}{\rot F}_{\Lt(\om)}
+\scp{\div E}{\div F}_{\Lt(\om)}
-\bscp{E_{\nor}}{(\div\nu)F_{\nor}}_{\Lt(\ga)}.
\intertext{If $E,F\in\Hc(\div,\om)$, then}
\scp{\na E}{\na F}_{\Lt(\om)}
&=\scp{\rot E}{\rot F}_{\Lt(\om)}
+\scp{\div E}{\div F}_{\Lt(\om)}
-\bscp{E_{\tan}}{(\na\nu)F_{\tan}}_{\Lt(\ga)}.
\intertext{In particular, if $E\in\Hc(\rot,\om)$, then}
\norm{\na E}_{\Lt(\om)}^{2}
&=\norm{\rot E}_{\Lt(\om)}^{2}
+\norm{\div E}_{\Lt(\om)}^{2}
-\int_{\ga}\div\nu\,|E_{\nor}|^{2},
\intertext{and, if $E\in\Hc(\div,\om)$, then}
\norm{\na E}_{\Lt(\om)}^{2}
&=\norm{\rot E}_{\Lt(\om)}^{2}
+\norm{\div E}_{\Lt(\om)}^{2}
-\int_{\ga}E_{\tan}\cdot(\na\nu E_{\tan}).
\end{align*}
\end{cor}

\begin{proof}
For $E\in\H^{1}(\om)\cap\Hc(\rot,\om)$
we have $\varphi E\times\nu\in\H^{1}(\om)$
and $E_{\tan}=0$ on $\ga$.
Moreover, for all $\Psi\in\Ci(\rt)$ 
we compute by Gau{\ss}' theorem
\begin{align*}
\scp{\varphi E\times\nu}{\rot\Psi}_{\Lt(\om)}
&=\bscp{\rot(\varphi E\times\nu)}{\Psi}_{\Lt(\om)}
-\int_{\om}\div(\varphi E\times\nu\times\Psi),\\
\int_{\om}\div(\varphi E\times\nu\times\Psi)
&=\int_{\ga}\nu\cdot(\varphi E\times\nu\times\Psi)
=\int_{\ga}(\nu\times E\times\nu)\cdot\Psi
=\int_{\ga}E_{\tan}\cdot\Psi
=0,
\end{align*}
which shows $\varphi E\times\nu\in\Hc(\rot,\om)$.
Hence, as $\na(\varphi F_{\nor})\in N(\rot_{\om})\subset\H(\rot,\om)$,
$$\bscp{\rot(\varphi E\times\nu)}{\na(\varphi F_{\nor})}_{\Lt(\om)}=0,$$
and Lemma \ref{lem:ibpsmooth} shows the first assertion for 
$E,F\in\H^{1}(\om)\cap\Hc(\rot,\om)$.

For $E\in\H^{1}(\om)\cap\Hc(\div,\om)$
we have $\varphi E_{\nor}\in\H^{1}(\om)$
and $E_{\nor}=0$ on $\ga$.
Moreover, for all $\Psi\in\Ci(\rt)$ 
we compute by Gau{\ss}' theorem
\begin{align*}
\scp{\varphi E_{\nor}}{\div\Psi}_{\Lt(\om)}
&=-\bscp{\na(\varphi E_{\nor})}{\Psi}_{\Lt(\om)}
+\int_{\om}\div(\varphi E_{\nor}\cdot\Psi),\\
\int_{\om}\div(\varphi E_{\nor}\cdot\Psi)
&=\int_{\ga}\nu\cdot(\varphi E_{\nor}\cdot\Psi)
=\int_{\ga}E_{\nor}\Psi_{\nor}
=0,
\end{align*}
which shows $\varphi E_{\nor}\in\Hc^{1}(\om)$.
Hence, as $\rot(\varphi F\times\nu)\in N(\div_{\om})\subset\H(\div,\om)$,
$$\bscp{\na(\varphi E_{\nor})}{\rot(\varphi F\times\nu)}_{\Lt(\om)}=0,$$
and Lemma \ref{lem:ibpsmooth} shows the second assertion for 
$E,F\in\H^{1}(\om)\cap\Hc(\div,\om)$.
\end{proof}

In the remaining part of this section, 
we focus on the implication of the presented integration by parts formula for convex geometries.

\begin{cor}[integration by parts for smooth convex domains and homogeneous boundary conditions]
\label{cor:ibpsmoothconvex}
Let $\om$ be bounded, smooth, and convex.
Then:
$$\forall\;E\in\H^{1}(\om)\cap\big(\Hc(\rot,\om)\cup\Hc(\div,\om)\big)
\qquad
\norm{\na E}_{\Lt(\om)}^{2}
\leq\norm{\rot E}_{\Lt(\om)}^{2}
+\norm{\div E}_{\Lt(\om)}^{2}$$
\end{cor}

\begin{proof}
As $\om$ is convex, $\na\nu$ and $\div\nu$ are non-negative,
as mentioned in the beginning of Section \ref{sec:ibpsd}.
Corollary \ref{cor:ibpsmooth} shows the assertion.
\end{proof}

\begin{ex}[unit ball]
\label{sec:unitball}
Let $\om=B_{3}\coloneqq B(0,1)\subset\rt$ be the Euclidean unit ball with boundary $\ga=S_{2}$.
Then for $x\neq0$
\begin{align*}
\nu(x)
&=\frac{x}{|x|},
&
\na\nu(x)
&=\frac{1}{|x|^{3}}
\begin{bmatrix}
x_{2}^{2}+x_{3}^{2}&-x_{1}x_{2}&-x_{1}x_{3}\\
-x_{2}x_{1}&x_{1}^{2}+x_{3}^{2}&-x_{2}x_{3}\\
-x_{3}x_{1}&-x_{3}x_{2}&x_{1}^{2}+x_{2}^{2}
\end{bmatrix}
\geq0,\\
\rot\nu(x)
&=0,
&
\div\nu(x)
&=\tr\na\nu(x)
=\frac{2}{|x|}
>0.
\end{align*}
Hence Corollary \ref{cor:ibpsmooth}, cf.~\eqref{eq:ibp1}, \eqref{eq:ibpIga}, and Lemma \ref{lem:ibpsmooth},
shows, e.g., for $E\in\H^{1}(B_{3})\cap\Hc(\rot,B_{3})$
\begin{align*}
\norm{\na E}_{\Lt(B_{3})}^{2}
&=\norm{\rot E}_{\Lt(B_{3})}^{2}
+\norm{\div E}_{\Lt(B_{3})}^{2}
-2\int_{S_{2}}|E_{\nor}|^{2}.
\end{align*}
Note that, e.g., for $E=\id$ we get
$3|B_{3}|=9|B_{3}|-2|S_{2}|$, i.e., the well-known result $|S_{2}|=3|B_{3}|$.
\end{ex}

\subsection{The Unit Cube}
\label{sec:ibpuc}

Finally, we address the prototype example for geometries with non-curved faces. 
The example deals with the geometric setting and the subsequent lemma proves 
the integration by parts formula in this geometry. 
The remarkable fact is that no curvature term appears and, thus, the Gaffney estimate becomes a mere equality.

Let $\om=Q\coloneqq (0,1)^{3}\subset\rt$ be the unit cube with boundary and faces
$$\ga=\bigcup_{k=1}^{3}\big(\ga_{k,+}\cup\ga_{k,-}\big),
\qquad
\ga_{k,\pm}\coloneqq\{x\in\ol{Q}:2x_{k}=1\pm1\},$$
and (almost everywhere defined) outward unit normal $\nu$ given by 
$$\nu|_{\ga_{k,\pm}}=:\nu^{k,\pm}=\pm e^{k}.$$
Note that $\nu_{l}|_{\ga_{k,\pm}}=\nu^{k,\pm}_{l}=\pm\delta_{lk}$.
Then for smooth vector fields $E$
the tangential and normal components, e.g., on $\ga_{3,\pm}$ are simply
$$E_{\nor}=\pm e^{3}\cdot E=\pm E_{3},
\quad
E_{\tan}=E-E_{3}=\tvec{E_{\parallel}}{0},
\qquad
E_{\parallel}\coloneqq\tvec{E_{1}}{E_{2}},
\quad
E_{\bot}\coloneqq E_{\nor}.$$
Hence, together with the surface gradient and divergence, 
we have on $\ga_{k,\pm}$ 
\begin{align*}
E_{\parallel}&\coloneqq E_{\parallel,k}\coloneqq\tvec{E_{n}}{E_{m}},
&
\div_{\parallel}E_{\parallel}&\coloneqq\div_{\parallel,k}E_{\parallel,k}\coloneqq\p_{n}E_{n}+\p_{m}E_{m},\\
E_{\bot}&\coloneqq E_{\bot,k}\coloneqq\pm E_{k},
&
\na_{\bot}E_{\bot}&\coloneqq\na_{\bot,k}E_{\bot,k}\coloneqq\tvec{\p_{n}E_{\bot}}{\p_{m}E_{\bot}}
=\pm\tvec{\p_{n}E_{k}}{\p_{m}E_{k}}
\end{align*}
for $\{k,n,m\}=\{1,2,3\}$, $n<m$.


\begin{lem}[integration by parts for the unit cube]
\label{lem:ibpcube}
Let $E,F\in\Ci(\rt)$. Then
$$\scp{\na E}{\na F}_{\Lt(Q)}
=\scp{\rot E}{\rot F}_{\Lt(Q)}
+\scp{\div E}{\div F}_{\Lt(Q)}
+\wt{\cI}_{\ga}(E,F),$$
where $\wt{\cI}_{\ga}$ is a boundary integral given by
$$\wt{\cI}_{\ga}(E,F)
\coloneqq\scp{E_{\parallel}}{\na_{\bot}F_{\bot}}_{\Lt(\ga)}
-\scp{E_{\bot}}{\div_{\parallel}F_{\parallel}}_{\Lt(\ga)}.$$
\end{lem}

\begin{proof}
By \eqref{eq:ibp1} we just have to compute 
\begin{align*}
\cI_{\ga}(E,F)
&=\sum_{k,l=1}^{3}
\int_{\ga}\Big(\nu_{l}E_{k}\p_{k}F_{l}-\nu_{k}E_{k}\p_{l}F_{l}\Big)\\
&\;=\sum_{k=1}^{3}\sum_{k\neq l=1}^{3}\sum_{j=1}^{3}
\int_{\ga_{j,\pm}}\Big(\nu^{j,\pm}_{l}E_{k}\p_{k}F_{l}-\nu^{j,\pm}_{k}E_{k}\p_{l}F_{l}\Big)\\
&\;=\sum_{k=1}^{3}\sum_{k\neq l=1}^{3}
\Big(\pm\int_{\ga_{l,\pm}}E_{k}\p_{k}F_{l}
\mp\int_{\ga_{k,\pm}}E_{k}\p_{l}F_{l}\Big)\\
&\;=\pm\scp{E_{1}}{\p_{1}F_{2}}_{\Lt(\ga_{2,\pm})}\pm\scp{E_{1}}{\p_{1}F_{3}}_{\Lt(\ga_{3,\pm})}
\mp\scp{E_{1}}{\p_{2}F_{2}+\p_{3}F_{3}}_{\Lt(\ga_{1,\pm})}\\
&\qquad
\pm\scp{E_{2}}{\p_{2}F_{1}}_{\Lt(\ga_{1,\pm})}\pm\scp{E_{2}}{\p_{2}F_{3}}_{\Lt(\ga_{3,\pm})}
\mp\scp{E_{2}}{\p_{1}F_{1}+\p_{3}F_{3}}_{\Lt(\ga_{2,\pm})}\\
&\qquad\qquad
\pm\scp{E_{3}}{\p_{3}F_{1}}_{\Lt(\ga_{1,\pm})}\pm\scp{E_{3}}{\p_{3}F_{2}}_{\Lt(\ga_{2,\pm})}
\mp\scp{E_{3}}{\p_{1}F_{1}+\p_{2}F_{2}}_{\Lt(\ga_{3,\pm})}\\
&\;=\pm\Bscp{\tvec{E_{2}}{E_{3}}}{\na_{2,3}F_{1}}_{\Lt(\ga_{1,\pm})}
\mp\Bscp{E_{1}}{\div_{2,3}\tvec{F_{2}}{F_{3}}}_{\Lt(\ga_{1,\pm})}\\
&\qquad
\pm\Bscp{\tvec{E_{1}}{E_{3}}}{\na_{1,3}F_{2}}_{\Lt(\ga_{2,\pm})}
\mp\Bscp{E_{2}}{\div_{1,3}\tvec{F_{1}}{F_{3}}}_{\Lt(\ga_{2,\pm})}\\
&\qquad\qquad
\pm\Bscp{\tvec{E_{1}}{E_{2}}}{\na_{1,2}F_{3}}_{\Lt(\ga_{3,\pm})}
\mp\Bscp{E_{3}}{\div_{1,2}\tvec{F_{1}}{F_{2}}}_{\Lt(\ga_{3,\pm})}\\
&=\sum_{k=1}^{3}
\Big(\scp{E_{\parallel}}{\na_{\bot}F_{\bot}}_{\Lt(\ga_{k,\pm})}
-\scp{E_{\bot}}{\div_{\parallel}F_{\parallel}}_{\Lt(\ga_{k,\pm})}\Big)
=\wt{\cI}_{\ga}(E,F),
\end{align*}
completing the proof.
\end{proof}

\section{Regularity and Gaffney's Inequality for Convex Domains}
\label{sec:gaffneyineq}

The following result is rooted in \cite[Lemma 2.1]{S1982a} apparently  due to discussions with Rolf Leis. 
It has also been used in higher-dimensional situations, see \cite[Lemma 3.9]{W2018}.

\begin{lem}
\label{lem:densityJukka}
Let $\H_{0}$, $\H_{1}$, and $\H_{2}$ be Hilbert spaces.
Moreover, let $\A_{0}:D(\A_{0})\subset\H_{0}\to\H_{1}$
and $\A_{1}:D(\A_{1})\subset\H_{1}\to\H_{2}$
be two densely defined and closed linear operators
satisfying the complex property $R(\A_{0})\subset N(\A_{1})$.
Let
$$\P_{\A_{0}}\coloneqq\id-\A_{0}(\A_{0}^{*}\A_{0}+1)^{-1}\A_{0}^{*}:D(\A_{0}^{*})\to D(\A_{0}^{*})$$
and $\cD\subset D_{1,0}\coloneqq D(\A_{1})\cap D(\A_{0}^{*})$.
Then:
\begin{itemize}
\item[\bf(i)] 
$\P_{\A_{0}}(D_{1,0})\subset D_{1,0}$,
i.e., $D_{1,0}$ is invariant under $\P_{\A_{0}}$.
\item[\bf(ii)] 
If $\cD$ is dense in $D(\A_{1})$, 
then $\P_{\A_{0}}(\cD)$ is dense in $D_{1,0}$.
\item[\bf(iii)] 
For $y\in D_{1,0}$, we have
$\norm{\P_{\A_{0}}y}_{D_{1,0}}\leq\norm{y}_{D(\A_{1})}$.
\item[\bf(iii')] 
If $D_{1,0}\ni y_{n}\to y\in D_{1,0}$ in $D(\A_{1})$,
then $D_{1,0}\ni\P_{\A_{0}}y_{n}\to\P_{\A_{0}}y\in D_{1,0}$ in $D_{1,0}$.
\end{itemize}
\end{lem}

Here, $D(\A_{1})$, $D(\A_{0}^{*})$, and $D_{1,0}$ are endowed with the graph inner products
\begin{align*}
\scp{\,\cdot\,}{\,\cdot\,}_{D(\A_{1})}
&\coloneqq\scp{\,\cdot\,}{\,\cdot\,}_{\H_{1}}
+\scp{\A_{1}\,\cdot\,}{\A_{1}\,\cdot\,}_{\H_{2}},\\
\scp{\,\cdot\,}{\,\cdot\,}_{D(\A_{0}^{*})}
&\coloneqq\scp{\,\cdot\,}{\,\cdot\,}_{\H_{1}}
+\scp{\A_{0}^{*}\,\cdot\,}{\A_{0}^{*}\,\cdot\,}_{\H_{0}},\\
\scp{\,\cdot\,}{\,\cdot\,}_{D_{1,0}}
&\coloneqq\scp{\,\cdot\,}{\,\cdot\,}_{\H_{1}}
+\scp{\A_{1}\,\cdot\,}{\A_{1}\,\cdot\,}_{\H_{2}}
+\scp{\A_{0}^{*}\,\cdot\,}{\A_{0}^{*}\,\cdot\,}_{\H_{0}},
\end{align*}
and the Hilbert space adjoints are given by 
$\A_{0}^{*}:D(\A_{0}^{*})\subset\H_{1}\to\H_{0}$
and $\A_{1}^{*}:D(\A_{1}^{*})\subset\H_{2}\to\H_{1}$.

\begin{proof}
Note that by the Riesz' representation theorem 
$\A_{0}^{*}\A_{0}+1:D(\A_{0}^{*}\A_{0})\to\H_{0}$ is a topological isomorphism.
Hence $(\A_{0}^{*}\A_{0}+1)^{-1}\big(R(\A_{0}^{*})\big)\subset D(\A_{0}^{*}\A_{0})$ and 
$$\A_{0}(\A_{0}^{*}\A_{0}+1)^{-1}\big(R(\A_{0}^{*})\big)\subset D(\A_{0}^{*})\cap N(\A_{1})$$
by the complex property.
Thus, $\P_{\A_{0}}(D_{1,0})\subset D_{1,0}$, 
showing (i).

Before we turn to the actual proof of (ii), 
we establish the following equality first:
\begin{equation}
\label{eq:proofJukkacomp}
\forall y,z\in D_{1,0}\,
\qquad
\scp{z}{\P_{\A_{0}}y}_{D_{1,0}}=\scp{z}{y}_{D(\A_{1})}.
\end{equation} 
Indeed, let $y,z\in D_{1,0}$ and put 
$\P_{\A_{0}}y=y-\A_{0}(\A_{0}^{*}\A_{0}+1)^{-1}\A_{0}^{*}y\in D_{1,0}$. 
Then
\begin{align*}
\A_{1}\P_{\A_{0}}y
&=\A_{1}y,\\
\A_{0}^{*}\P_{\A_{0}}y
&=\A_{0}^{*}y-\A_{0}^{*}\A_{0}(\A_{0}^{*}\A_{0}+1)^{-1}\A_{0}^{*}y
=(\A_{0}^{*}\A_{0}+1)^{-1}\A_{0}^{*}y,\\
\scp{\A_{0}^{*}z}{\A_{0}^{*}\P_{\A_{0}}y}_{\H_{0}}
&=\bscp{z}{\A_{0}(\A_{0}^{*}\A_{0}+1)^{-1}\A_{0}^{*}y}_{\H_{1}}
=\bscp{z}{(1-\P_{\A_{0}})y}_{\H_{1}},
\intertext{and it follows that}
\scp{z}{\P_{\A_{0}}y}_{D_{1,0}}
&=\scp{z}{\P_{\A_{0}}y}_{\H_{1}}
+\scp{\A_{1}z}{\A_{1}\P_{\A_{0}}y}_{\H_{2}}
+\scp{\A_{0}^{*}z}{\A_{0}^{*}\P_{\A_{0}}y}_{\H_{0}}\\
&=\scp{z}{y}_{\H_{1}}
+\scp{\A_{1}z}{\A_{1}y}_{\H_{2}}
=\scp{z}{y}_{D(\A_{1})}.
\end{align*}

For the proof of (ii), let $\cD$ be dense in $D(\A_{1})$ 
and take $z\in D_{1,0}\cap\big(\P_{\A_{0}}(\cD)\big)^{\bot_{D_{1,0}}}$.
Then, for all $\P_{\A_{0}}y\in\P_{\A_{0}}(\cD)\subset D_{1,0}$
with $y\in\cD$, using \eqref{eq:proofJukkacomp}, we get
$$0=\scp{z}{\P_{\A_{0}}y}_{D_{1,0}}=\scp{z}{y}_{D(\A_{1})},$$
and, as $\cD$ is dense in $D(\A_{1})$, we conclude $z=0$.

For (iii) and (iii'), let $y\in D_{1,0}$.
Then by \eqref{eq:proofJukkacomp} with $z=\P_{\A_{0}}y$
$$\norm{\P_{\A_{0}}y}_{D_{1,0}}^{2}
=\scp{\P_{\A_{0}}y}{y}_{D(\A_{1})}
\leq\norm{\P_{\A_{0}}y}_{D(\A_{1})}\norm{y}_{D(\A_{1})}
\leq\norm{\P_{\A_{0}}y}_{D_{1,0}}\norm{y}_{D(\A_{1})},$$
i.e., $\norm{\P_{\A_{0}}y}_{D_{1,0}}\leq\norm{y}_{D(\A_{1})}$.
\end{proof}

\subsection{Bounded Domains}
\label{sec:gaffneybdom}

The latter density result may now be used to prove Gaffney's inequality in the smooth bounded domain case.

\begin{lem}[Gaffney's inequality for bounded, smooth, and convex domains]
\label{lem:gaffneyconvexsmooth}
Let $\om$ be bounded, smooth, and convex.
Then $\om$ is a Gaffney domain.

More precisely: 
If $E\in\Hc(\rot,\om)\cap\H(\div,\om)$
or $E\in\H(\rot,\om)\cap\Hc(\div,\om)$,
then $E\in\H^{1}(\om)$ and 
$$\norm{\na E}_{\Lt(\om)}^{2}
\leq\norm{\rot E}_{\Lt(\om)}^{2}
+\norm{\div E}_{\Lt(\om)}^{2}.$$
\end{lem}

\begin{proof} 

Let $E\in\Hc(\rot,\om)\cap\H(\div,\om)=:D_{1,0}$.
We use Lemma \ref{lem:densityJukka}(iii) 
for $\A_{0}\coloneqq\nac$, $\A_{1}\coloneqq\rotc$ and $\cD\coloneqq\Cic(\om)$. 
Then $\A_{0}^{*}=-\div$ and $\A_{1}^{*}\coloneqq\rot$.
As $\cD$ is dense in $D(\A_{1})=\Hc(\rot,\om)$,
there exists $(E_{\ell})\subset\cD$ 
such that 
$$\wt{E}_{\ell}
\coloneqq\P_{\A_{0}}E_{\ell}
=E_{\ell}-\A_{0}(\A_{0}^{*}\A_{0}+1)^{-1}\A_{0}^{*}E_{\ell}
=E_{\ell}+\nac(1-\div\nac)^{-1}\div E_{\ell}
\to E
\quad\text{in }D_{1,0}.$$
Elliptic regularity for the Dirichlet Laplacian $\A_{0}^{*}\A_{0}=-\div\nac$
yields $(1-\div\nac)^{-1}\div E_{\ell}\in\H^{2}(\om)$,
and thus $\wt{E}_{\ell}\in\H^{1}(\om)\cap D_{1,0}=\H^{1}(\om)\cap\Hc(\rot,\om)$.
Corollary \ref{cor:ibpsmoothconvex} shows
$$\bnorm{\na(\wt{E}_{\ell}-\wt{E}_{k})}_{\Lt(\om)}^{2}
\leq\bnorm{\rot(\wt{E}_{\ell}-\wt{E}_{k})}_{\Lt(\om)}^{2}
+\bnorm{\div(\wt{E}_{\ell}-\wt{E}_{k})}_{\Lt(\om)}^{2}
\leq\norm{\wt{E}_{\ell}-\wt{E}_{k}}_{D_{1,0}}^{2}.$$
Thus $(\wt{E}_{\ell})$ is a Cauchy sequence in $\H^{1}(\om)$.
Hence $\wt{E}_{\ell}\to E$ in $\H^{1}(\om)$, in particular, $E\in\H^{1}(\om)$.
Moreover, again by Corollary \ref{cor:ibpsmoothconvex}
$$\norm{\na E}_{\Lt(\om)}^{2}
\ot\norm{\na\wt{E}_{\ell}}_{\Lt(\om)}^{2}
\leq\norm{\rot\wt{E}_{\ell}}_{\Lt(\om)}^{2}
+\norm{\div\wt{E}_{\ell}}_{\Lt(\om)}^{2}
\to\norm{\rot E}_{\Lt(\om)}^{2}
+\norm{\div E}_{\Lt(\om)}^{2}.$$

Analogously, we prove the assertions for $E\in\H(\rot,\om)\cap\Hc(\div,\om)$
using Lemma \ref{lem:densityJukka} with $\A_{0}=\na$, $\A_{1}=\rot$,
and $\A_{0}^{*}=-\divc$ and $\A_{1}^{*}\coloneqq\rotc$,
as well as elliptic regularity for the Neumann Laplacian $\A_{0}^{*}\A_{0}=-\divc\na$.
\end{proof}

It is well-known that smoothness of the considered bounded domain can be dropped:

\begin{theo}[Gaffney's inequality for bounded and convex domains]
\label{theo:gaffneyconvexbd}
Let $\om\subset\rt$ be bounded and convex.
Then $\om$ is a Gaffney domain.

More precisely:
If $E\in\Hc(\rot,\om)\cap\H(\div,\om)$
or $E\in\H(\rot,\om)\cap\Hc(\div,\om)$,
then $E\in\H^{1}(\om)$ and 
$$\norm{\na E}_{\Lt(\om)}^{2}
\leq\norm{\rot E}_{\Lt(\om)}^{2}
+\norm{\div E}_{\Lt(\om)}^{2}.$$
\end{theo}

For a proof see the book of Grisvard, 
cf.~\cite[Theorem 3.2.1.2, Theorem 3.2.1.3]{G1985a},
or \cite[Corollary 3.6, Theorem 3.9]{GR1986a}
and \cite[Theorem 2.17]{ABDG1998a}\footnote{We note 
that in \cite[p. 834]{ABDG1998a}
the proof for $X_{N}(\om)$ appears to be wrong. 
In fact, due to the solenoidal condition, 
one needs to use the space $X_{T}(\om_{k})$ instead of $V_{T}(\om_{k})$. 
However, in $X_{T}(\om_{k})$, the arguments for the second order elliptic system for $\zeta$ no longer hold. 
The present approach resolves these inconsistencies.} 
for the case of Maxwell's equations.
Our proof, following the book of Grisvard \cite{G1985a},
avoids the misleading notion of traces and the uniqueness of solutions of second order elliptic systems.
A generalised version has already been presented in the appendix of \cite{P2019a}.
Here, we sketch the proof only, 
and provide a detailed version in Appendix \ref{sec:detproofgaffneyconvex}.

\begin{proof}[Proof of Theorem \ref{theo:gaffneyconvexbd}]
Let $E\in\Hc(\rot,\om)\cap\H(\div,\om)$.
We pick a sequence of increasing, convex, and smooth subdomains 
$\om_{\ell}\subset\ol{\om}_{\ell}\subset\om_{\ell+1}\subset\dots\subset\om$
such that $\dist(\p\om,\p\om_{\ell})\to0$,
see, e.g., \cite[Lemma 3.2.1.1]{G1985a}. 
For $\om_{\ell}$ we find
$H_{\ell}\in\H(\rot,\om_{\ell})$ 
such that for all $\Psi\in\H(\rot,\om_{\ell})$
\begin{align}
\label{eq:pdedefomn}
\scp{H_{\ell}}{\Psi}_{\H(\rot,\om_{\ell})}
=\scp{E}{\rot\Psi}_{\Lt(\om_{\ell})}
-\scp{\rot E}{\Psi}_{\Lt(\om_{\ell})}
\end{align}
(Riesz isometry). 
Then 
$$E_{\ell}\coloneqq E-\rot H_{\ell}\in\Hc(\rot,\om_{\ell})\cap\H(\div,\om_{\ell}),
\qquad
\rot E_{\ell}=\rot E+H_{\ell},
\quad
\div E_{\ell}=\div E.$$
By Lemma \ref{lem:gaffneyconvexsmooth}
we have $E_{\ell}\in\H^{1}(\om_{\ell})$ with
\begin{align}
\label{eq:naomeganest}
\norm{\na E_{\ell}}_{\Lt(\om_{\ell})}^{2}
&\leq\norm{\rot E_{\ell}}_{\Lt(\om_{\ell})}^{2}
+\norm{\div E_{\ell}}_{\Lt(\om_{\ell})}^{2}
=\norm{\rot E+H_{\ell}}_{\Lt(\om_{\ell})}^{2}
+\norm{\div E}_{\Lt(\om_{\ell})}^{2}.
\end{align}
For $\Psi=H_{\ell}$, \eqref{eq:pdedefomn} shows
\begin{align}
\label{eq:zetanest}
\norm{H_{\ell}}_{\H(\rot,\om_{\ell})}^{2}
&=\scp{E}{\rot H_{\ell}}_{\Lt(\om_{\ell})}
-\scp{\rot E}{H_{\ell}}_{\Lt(\om_{\ell})}
\leq\norm{E}_{\H(\rot,\om_{\ell})}\norm{H_{\ell}}_{\H(\rot,\om_{\ell})}
\end{align}
and thus
\begin{align}
\label{eq:zetanesttwo}
\norm{H_{\ell}}_{\H(\rot,\om_{\ell})}
&\leq\norm{E}_{\H(\rot,\om_{\ell})}
\leq\norm{E}_{\H(\rot,\om)}.
\end{align}
Combining \eqref{eq:naomeganest} and the equation part of \eqref{eq:zetanest} we observe
\begin{align*}
\norm{E_{\ell}}_{\H^{1}(\om_{\ell})}^{2}
&=\norm{E_{\ell}}_{\Lt(\om_{\ell})}^{2}
+\norm{\na E_{\ell}}_{\Lt(\om_{\ell})}^{2}\\
&\leq\norm{E_{\ell}}_{\Lt(\om_{\ell})}^{2}
+\norm{\rot E+H_{\ell}}_{\Lt(\om_{\ell})}^{2}
+\norm{\div E}_{\Lt(\om_{\ell})}^{2}\\
&=\norm{E}_{\Lt(\om_{\ell})}^{2}
+\norm{\rot H_{\ell}}_{\Lt(\om_{\ell})}^{2}
+\norm{\rot E}_{\Lt(\om_{\ell})}^{2}
+\norm{H_{\ell}}_{\Lt(\om_{\ell})}^{2}
+\norm{\div E}_{\Lt(\om_{\ell})}^{2}\\
&\qquad\qquad
-2\scp{E}{\rot H_{\ell}}_{\Lt(\om_{\ell})}
+2\scp{\rot E}{H_{\ell}}_{\Lt(\om_{\ell})}\\
&=\norm{E}_{\H(\rot,\om_{\ell})\cap\H(\div,\om_{\ell})}^{2}
-\norm{H_{\ell}}_{\H(\rot,\om_{\ell})}^{2},
\end{align*}
and therefore
\begin{align}
\label{eq:hoomeganest}
\norm{E_{\ell}}_{\H^{1}(\om_{\ell})}
\leq\norm{E}_{\H(\rot,\om_{\ell})\cap\H(\div,\om_{\ell})}
\leq\norm{E}_{\H(\rot,\om)\cap\H(\div,\om)}.
\end{align}

Let us denote the extension by zero to $\om$ by $\wt{\cdot}$.
Then by \eqref{eq:zetanesttwo} and \eqref{eq:hoomeganest}
the sequences $(\wt{H}_{\ell})$, $(\wt{\rot H}_{\ell})$,
and $(\wt{E}_{\ell})$, $(\wt{\na E}_{\ell})$
are bounded in $\Lt(\om)$, 
and we can extract weakly converging subsequences, again denoted by the index $\ell$, such that
\begin{align*}
\wt{H}_{\ell}
&\wto{\Lt(\om)}H\in\Lt(\om),
&
\wt{E}_{\ell}
&\wto{\Lt(\om)}\wh{E}\in\Lt(\om),\\
(\wt{\rot H}_{\ell})
&\wto{\Lt(\om)}F\in\Lt(\om),
&
\wt{\na E}_{\ell}
&\wto{\Lt(\om)}G\in\Lt(\om).
\end{align*}
Then $\wh{E}\in\H^{1}(\om)$ and $\na\wh{E}=G$
as well as $H\in\H(\rot,\om)$ and $\rot H=F$.
Moreover, we have for $\Psi\in\H(\rot,\om)\subset\H(\rot,\om_{\ell})$
\begin{align*}
\scp{H_{\ell}}{\Psi}_{\H(\rot,\om_{\ell})}
=\scp{\wt{H}_{\ell}}{\Psi}_{\Lt(\om)}
+\scp{\wt{\rot H}_{\ell}}{\rot\Psi}_{\Lt(\om)}
\to\scp{H}{\Psi}_{\H(\rot,\om)}
\end{align*}
and, by \eqref{eq:pdedefomn},
\[
 \scp{H_{\ell}}{\Psi}_{\H(\rot,\om_{\ell})} =\scp{E}{\rot\Psi}_{\Lt(\om_{\ell})}
-\scp{\rot E}{\Psi}_{\Lt(\om_{\ell})}
\to\scp{E}{\rot\Psi}_{\Lt(\om)}
-\scp{\rot E}{\Psi}_{\Lt(\om)}
=0
\]
as $E\in\Hc(\rot,\om)$.
For $\Psi=H$ we get $H=0$.
Furthermore, we observe that on the one hand by \eqref{eq:hoomeganest}
\begin{align*}
\scp{\wh{E}}{\wt{E}_{\ell}}_{\Lt(\om)}
+\scp{\na\wh{E}}{\wt{\na E}_{\ell}}_{\Lt(\om)}
\to\scp{\wh{E}}{\wh{E}}_{\Lt(\om)}
+\scp{\na\wh{E}}{\na\wh{E}}_{\Lt(\om)}
=\norm{\wh{E}}_{\H^{1}(\om)}^{2}
\end{align*}
and, by \eqref{eq:hoomeganest}, on the other hand
\begin{align*}
\scp{\wh{E}}{\wt{E}_{\ell}}_{\Lt(\om)}
+\scp{\na\wh{E}}{\wt{\na E}_{\ell}}_{\Lt(\om)} & =\scp{\wh{E}}{E_{\ell}}_{\Lt(\om_{\ell})}
+\scp{\na\wh{E}}{\na E_{\ell}}_{\Lt(\om_{\ell})}\\&
\leq\norm{\wh{E}}_{\H^{1}(\om_{\ell})}
\norm{E_{\ell}}_{\H^{1}(\om_{\ell})}
\leq\norm{\wh{E}}_{\H^{1}(\om)}
\norm{E}_{\H(\rot,\om)\cap\H(\div,\om)},
\end{align*}
showing
\begin{align}
\label{eq:hoomnormomegahat}
\norm{\wh{E}}_{\H^{1}(\om)}
\leq\norm{E}_{\H(\rot,\om)\cap\H(\div,\om)}.
\end{align}

Since
$E=E_{\ell}+\rot H_{\ell}$ in $\om_{\ell}$,
we have
$\chi_{\om_{\ell}}E
=\wt{E}_{\ell}+\wt{\rot H}_{\ell}
\wto{\Lt(\om)}\wh{E}+\rot H
=\wh{E}$.
In any case, 
$\chi_{\om_{\ell}}E\to E$ in $\Lt(\om)$. Thus
$E=\wh{E}\in\H^{1}(\om)$ and by \eqref{eq:hoomnormomegahat}
$$\norm{E}_{\H^{1}(\om)}
\leq\norm{E}_{\H(\rot,\om)\cap\H(\div,\om)},$$
in particular,
$\norm{\na E}_{\Lt(\om)}^{2}
\leq\norm{\rot E}_{\Lt(\om)}^{2}
+\norm{\div E}_{\Lt(\om)}^{2}$, taking the limit in \eqref{eq:naomeganest}.

Similarly, we show the assertions for 
$E\in\H(\rot,\om)\cap\Hc(\div,\om)$; which is carried out in Appendix \ref{sec:detproofgaffneyconvex}.
\end{proof}

\subsection{Possibly Unbounded Domains}
\label{sec:gaffneyunbdom}

Theorem \ref{theo:gaffneyconvexbd} from above
states that bounded and convex domains are  Gaffney domains.
The next result enables us to transition from bounded to possibly unbounded domains.

\begin{lem}[permanence principle for (exact) Gaffney domains]
\label{lem:approxG}
Let $\om^{\star}\subset\om\subset\rt$ be open, bounded, and star-shaped
with center $x_{\star}\in\om^{\star}$. 
Define $\om^{\star}_{r}\coloneqq x_{\star}+r(\om^{\star}-x_{\star})$ for all $r>0$.

If $\om_{r}\coloneqq\om^{\star}_{r}\cap\om$ is a (exact) Gaffney domain for all $r$, 
then so is $\om$. 
\end{lem}

\begin{proof} 
Let $E\in\Hc(\rot,\om)\cap\H(\div,\om)$
or $E\in\H(\rot,\om)\cap\Hc(\div,\om)$.
Without loss of generality, $x_{\star}=0$.
As $\om^{\star}_{1}=\om^{\star}$ is open, bounded, and star-shaped,
so is $\om^{\star}_{1/2}$.
Moreover, it is compactly contained in $\om^{\star}_{1}$.
Thus, there exists $\varphi\in\Cic\big(\rt,[0,1]\big)$ such that
$\varphi|_{\om^{\star}_{1/2}}=1$ and $\varphi|_{\rt\setminus\om^{\star}_{2/3}}=0$.
Put $\varphi_{r}\coloneqq\varphi(\,\cdot\,/r)$ for $r>0$.
Then $\varphi_{r}|_{\om^{\star}_{r/2}}=1$
and $\varphi_{r}|_{\rt\setminus\om^{\star}_{2r/3}}=0$.
Note that $\supp\na\varphi_{r}\subset\ol{\om}^{\star}_{2r/3}\setminus\om^{\star}_{r/2}$
and $|\na\varphi_{r}|\leq c/r$.
For all $r$ the product rules
\begin{align*}
\p_{j}(\varphi_{r}E)&=\varphi_{r}\p_{j}E+(\p_{j}\varphi_{r})E,\\
\rot(\varphi_{r}E)&=\varphi_{r}\rot E+(\na\varphi_{r})\times E,\\
\div(\varphi_{r}E)&=\varphi_{r}\div E+(\na\varphi_{r})\cdot E
\end{align*}
imply that 
$\varphi_{r}E\in\Hc(\rot,\om_{r})\cap\H(\div,\om_{r})$ 
or $\varphi_{r}E\in\H(\rot,\om_{r})\cap\Hc(\div,\om_{r})$. 
Since $\om_{r}$ is a Gaffney domain, we get
$\varphi_{r}E\in\H^{1}(\om_{r})$ with 
\begin{align}
\label{eq:estr1}
\bnorm{\na(\varphi_{r}E)}_{\Lt(\om)}^{2}
\leq\Big(
\bnorm{\rot(\varphi_{r}E)}_{\Lt(\om)}^{2}+\bnorm{\div(\varphi_{r}E)}_{\Lt(\om)}^{2}
\Big).
\end{align}
In particular, we have $E\in\H^{1}(\om_{r})$ for all $r$.
Then (with $c$ independent of $r$) by \eqref{eq:estr1}
\begin{align}
\begin{aligned}
\label{eq:estr2}
\norm{\varphi_{r}\na E}_{\Lt(\om)}
&\leq c\Big(\bnorm{\na(\varphi_{r}E)}_{\Lt(\om)}
+\sum_{j=1}^{3}\bnorm{(\p_{j}\varphi_{r})E}_{\Lt(\om)}\Big)\\
&\leq c\Big(\bnorm{\rot(\varphi_{r}E)}_{\Lt(\om)}
+\bnorm{\div(\varphi_{r}E)}_{\Lt(\om)}
+\frac{1}{r}\norm{E}_{\Lt(\om)}\Big)\\
&\leq c\Big(\bnorm{\rot E}_{\Lt(\om)}
+\bnorm{\div E}_{\Lt(\om)}
+\frac{1}{r}\norm{E}_{\Lt(\om)}\Big).
\end{aligned}
\end{align}
The monotone convergence theorem yields
$\norm{\na E}_{\Lt(\om)}
\leq c\big(\bnorm{\rot E}_{\Lt(\om)}
+\bnorm{\div E}_{\Lt(\om)}\big)$,
i.e., $E\in\H^{1}(\om)$.
Finally, again by the product rules and $|\na\varphi_{r}|\leq c/r$, 
we infer that $\na(\varphi_{r}E)\to\na E$, 
$\rot(\varphi_{r}E)\to\rot E$,
and $\div(\varphi_{r}E)\to\div E$ in $\Lt(\om)$. 
This, together with \eqref{eq:estr1}, 
implies \eqref{eq:estr1} for $\varphi_{r}E$ replaced by $E$, 
i.e., $\om$ is a Gaffney domain.

If \eqref{eq:estr1} holds with an equality,
then this transferred to $E$ as well.
\end{proof}

As convex domains are star-shaped,
an immediate implication of the previous lemma
and Theorem \ref{theo:gaffneyconvexbd}
is the following main result of this section:

\begin{theo}[Gaffney's inequality for convex domains]
\label{theo:gaffneyconvex}
Let $\om\subset\rt$ be convex.
Then $\om$ is a Gaffney domain.
More precisely:
If $E\in\Hc(\rot,\om)\cap\H(\div,\om)$
or $E\in\H(\rot,\om)\cap\Hc(\div,\om)$,
then $E\in\H^{1}(\om)$ and 
$$\norm{\na E}_{\Lt(\om)}^{2}
\leq\norm{\rot E}_{\Lt(\om)}^{2}
+\norm{\div E}_{\Lt(\om)}^{2}.$$
\end{theo}

Theorem \ref{theo:gaffneyconvex} implies results 
also for the harmonic Dirichlet fields $\calH_{\sfD}(\om)\coloneqq N(\rotc_{\om})\cap N(\div_{\om})$ 
and harmonic Neumann fields $\calH_{\sfN}(\om)\coloneqq N(\rot_{\om})\cap N(\divc_{\om})$
as well as for the Dirichlet and Neumann Laplacians.

\begin{cor}
\label{cor:harmfields}
If $\om\subset\rt$ is convex, then $\calH_{\sfD}(\om)$ and $\calH_{\sfN}(\om)$ are trivial.
\end{cor}

\begin{proof}
Let $E\in\calH_{\sfD}(\om)\cup\calH_{\sfN}(\om)$.
Then $E$ is constant by Theorem \ref{theo:gaffneyconvex}.
In either case, the respective boundary condition implies $E=0$.
\end{proof}

\begin{cor}
\label{cor:gaffneyconvex}
Let $\om\subset\rt$ be convex.
Moreover, let $u\in\Hc^{1}(\om)$ with $\na u\in\H(\div,\om)$
or $u\in\H^{1}(\om)$ with $\na u\in\Hc(\div,\om)$.
Then $u\in\H^{2}(\om)$ and 
$$\norm{\na\na u}_{\Lt(\om)}
\leq\norm{\Delta u}_{\Lt(\om)}.$$
\end{cor}

\begin{proof}
We observe $\na u\in\H(\rot,\om)\cap\Hc(\div,\om)$
or $\na u\in\Hc(\rot,\om)\cap\H(\div,\om)$ by the complex property.
Theorem \ref{theo:gaffneyconvex}
shows the result by setting $E=\na u$.
\end{proof}

\subsection{Exact Gaffney Domains}
\label{sec:exactgaffneydom}

Although not being relevant for our results,
we note the following facts about exactness:
Using the particular integration by parts result for cubes,
Lemma \ref{lem:ibpcube},
which can be generalised to polyhedrons
with some additional technical and notational efforts,
and a sophisticated investigation of the 
surface differential operators
$\na_{\bot}$ and $\div_{\parallel}$
(continuous extensions of them),
it is possible to give proper meaning to the boundary term
$$\wt{\cI}_{\ga}(E,F)
=\scp{E_{\parallel}}{\na_{\bot}F_{\bot}}_{\Lt(\ga)}
-\scp{E_{\bot}}{\div_{\parallel}F_{\parallel}}_{\Lt(\ga)}$$
even for vector fields $E,F$ belonging merely to $\H^{1}(\om)$,
see, e.g., \cite{C1991a}.
More precisely:

\begin{lem}[{ \cite[Theorem 4.1]{C1991a}}]
\label{lem:costabel}
Let $\om$ be a bounded polyhedron,
and let $E,F\in\H^{1}(\om)$. Then
$$\scp{\na E}{\na F}_{\Lt(Q)}
=\scp{\rot E}{\rot F}_{\Lt(Q)}
+\scp{\div E}{\div F}_{\Lt(Q)}
+\wh{\cI}_{\ga}(E,F),$$
where
$$\wh{\cI}_{\ga}(E,F)
\coloneqq\scp{E_{\tan}}{\na_{\tan}F_{\nor}}_{\Lt(\ga)}
-\scp{E_{\nor}}{\div_{\tan}F_{\tan}}_{\Lt(\ga)}$$
is given as sum over all faces of $\ga$.
\end{lem}

It turns out that $\wh{\cI}_{\ga}(E,F)$ still vanishes,
if $E,F$ belong additionally also to $\Hc(\rot,\om)$ or $\Hc(\div,\om)$, 
in particular, for $E\in\H^{1}(\om)\cap\Hc(\rot,\om)$
or $E\in\H^{1}(\om)\cap\Hc(\div,\om)$ we obtain
$$\norm{\na E}_{\Lt(\om)}^{2}
=\norm{\rot E}_{\Lt(\om)}^{2}
+\norm{\div E}_{\Lt(\om)}^{2}.$$
Hence, bounded and convex polyhedrons are exact Gaffney domains.
Note that the convexity is still needed for the regularity part of the result.

For the cube there is another elementary way:
By the tensor structure of a cube $Q$
products of sine and cosine functions,
as eigenfunctions of Laplacians on the unit interval,
form a complete orthonormal system in $\Lt(Q)$
yielding dense subsets for the different boundary conditions.
This, together with the density lemma
(the abstract result Lemma \ref{lem:densityJukka})
shows that cubes are exact Gaffney domains.

Together with Lemma \ref{lem:approxG} we conclude:

\begin{cor}[exact Gaffney domains]
\label{cor:exactgaffney}
All (possibly unbounded) convex polyhedrons are exact Gaffney domains.
In particular, $\rt$, $\reals^{2}\times(0,1)$, and $\reals\times(0,1)^{2}$
are exact Gaffney domains.
\end{cor}

\section{Closed Range Results for Cuboids}
\label{sec:clrancubes}

We finally turn to closed range results using Gaffney's inequality. 
In particular, we will carry out the strategy sketched in the introduction. 
Here we focus on cuboids and discuss the validity of a Friedrichs type estimate first.


Let
$-\infty\leq a_{j}<b_{j}\leq\infty$ for $j\in\{1,2,3\}$ and define
$$I_{j}\coloneqq(a_{j},b_{j})\subset\reals
\qquad
\ell_{j}\coloneqq b_{j}-a_{j},
\qquad
c_{j}\coloneqq\frac{\ell_{j}}{\sqrt{2}},$$
together with the (possibly infinite) \textbf{cuboids}  
\begin{equation}
\label{eq:q}
Q\coloneqq I_{1}\times I_{2}\times I_{3}\subset\rt.
\end{equation}
For $Q$ we define the \textbf{number of directions of boundedness}
$$\db_{Q}\coloneqq 
\#\big\{j\in\{1,2,3\}:\ell_{j}<\infty\big\}$$
with the usual convention $\infty-(-\infty)=\infty$.  

Instrumental in the proof of the closed range result is the following variant of Friedrichs' estimate.
For $\db_{Q}\geq1$ we have:

\begin{lem}[Friedrichs estimate]
\label{lem:friedrichs}
Let $-\infty<a_{3}<b_{3}<\infty$ and $u\in\H^{1}(Q)$ 
with $u|_{I_{1}\times I_{2}\times\{a_{3}\}}=0$. 
Then 
$$\norm{u}_{\Lt(Q)}\leq 
c_{3}\norm{\p_{3}u}_{\Lt(Q)}\leq 
c_{3}\norm{\na u}_{\Lt(Q)}.$$
\end{lem}

\begin{proof}
By a density argument, it suffices to establish the inequality for
$u\in\Cic\big(\reals^{2}\times(a_{3},\infty)\big)$.
Then
$$u(x_{1},x_{2},x_{3})=\int_{a_{3}}^{x_{3}}\p_{3}u(x_{1},x_{2},\,\cdot\,),
\qquad
x_{j}\in I_{j}.$$
Thus 
$\displaystyle\big|u(x_{1},x_{2},x_{3})\big|^{2}\leq
(x_{3}-a_{3})\int_{a_{3}}^{b_{3}}\big|\p_{3}u(x_{1},x_{2},\,\cdot\,)\big|^{2}$,
which implies 
\begin{align}
\label{eq:friedrichs1d}
\int_{a_{3}}^{b_{3}}\big|u(x_{1},x_{2},\,\cdot\,)\big|^{2}\leq
\frac{\ell_{3}^{2}}{2}\int_{a_{3}}^{b_{3}}\big|\p_{3}u(x_{1},x_{2},\,\cdot\,)\big|^{2}.
\end{align}
Integration over $I_{1}$, $I_{2}$ shows 
$\norm{u}_{\Lt(Q)}^{2}\leq c_{3}^{2}\norm{\p_{3}u}_{\Lt(Q)}^{2}$.
\end{proof}

Since our aim is to characterise closed range results in terms of directions of boundedness, 
we will also provide statements, when the range is not closed. 
The key for this line of arguments will be explicit constructions showing 
that a closed range inequality cannot hold. 
We will frequently use the following family of functions. For 
$n\in\n$ let $f_{n}\in\H^{1}(\reals)$ be given by
\begin{equation}
\label{eq:fn}
f_{n}(t)
\coloneqq
\begin{cases} 
t-1&\text{, if }1\leq t<2,\\
1&\text{, if }2\leq t<n,\\
1+n-t&\text{, if }n\leq t<n+1,\\
0&,\text{ else}.
\end{cases}
\end{equation}


%

\begin{lem}
\label{lem:notclosed}
Let $Q$ be as in \eqref{eq:q},
and let $\db_{Q}=0$. 
Then $R(\nac_{Q})$ and $R(\na_{Q})$ are not closed.
\end{lem}

\begin{proof}
Note that $N(\nac_{Q})=N(\na_{Q})=\{0\}$.
Without loss of generality, let
$\ell_{j}=b_{j}=\infty$ and $a_{j}\in\{-\infty,0\}$
for all $j\in\{1,2,3\}$.
We define $u_{n}\in\H^{1}(\rt)$ 
by $u_{n}(x)\coloneqq f_{n}(x_{1})f_{n}(x_{2})f_{n}(x_{3})$ 
with $f_{n}$ from \eqref{eq:fn}.
Then $u_{n}\in\Hc^{1}(Q)$ and 
$\norm{u_{n}}_{\Lt(Q)}^{2}\sim n^{3}$  
and $\norm{\na u_{n}}_{\Lt(Q)}^{2}\sim n^{2}$. 
Thus, a Friedrichs type estimate cannot hold
for $\nac$ or $\na$,
and $R(\nac_{Q})$ and $R(\na_{Q})$ are not closed.
\end{proof}

\begin{theo}[closed range of the gradient]
\label{theo:closedrangegrad} 
Let $Q$ be as in \eqref{eq:q}. 
Then:
$$R(\nac_{Q})\text{ closed}
\qequi
R(\div_{Q})\text{ closed}
\qequi 
\db_{Q}\geq1$$
\end{theo}

\begin{proof}
Banach's closed range theorem shows the first equivalence,
and Lemma \ref{lem:friedrichs} proves the main implication.
If $R(\nac_{Q})$ is closed,
then $\db_{Q}=0$ cannot hold by Lemma \ref{lem:notclosed}. 
\end{proof}

Thus, it suffices to consider $\rotc$ and $\divc$ in the following.

\begin{theo}[closed range of the rotation]
\label{theo:closedrangerot}
Let $Q$ be as in \eqref{eq:q}. 
Then:
$$R(\rotc_{Q})\text{ closed}
\qequi
R(\rot_{Q})\text{ closed}
\qequi
\db_{Q}\geq2$$
\end{theo}

\begin{proof}
Again, the closed range theorem yields the first equivalence. 
In order to show a closed range result for $\rotc$, by Theorem \ref{theo:fatblem1}, 
it suffices to find $c>0$ such that for all $E\in D(\rotc)\cap N(\rotc)^\bot$
$$\norm{E}_{\Lt(Q)}\leq 
c\norm{\rotc E}_{\Lt(Q)}.$$

Thus, let $E\in D(\rotc)\cap N(\rotc)^\bot$. 
As 
$N(\rotc)^\bot
=\ol{R(\rot)}
\subset N(\div)
\subset D(\div)$, 
we deduce that $E\in D(\rotc)\cap N(\div)$. 
By Theorem \ref{theo:gaffneyconvex} we infer $E\in\H^{1}(Q)$ and
\begin{equation}
\label{eq:gaffred}
\norm{\na E}_{\Lt(Q)}
\leq\big(\norm{\rotc E}_{\Lt(Q)}^{2} 
+\norm{\div E}_{\Lt(Q)}^{2}\big)^{1/2}
=\norm{\rotc E}_{\Lt(Q)}.
\end{equation}

Next, if $\db_{Q}\geq 2$, we may assume without loss of generality, that $\ell_{2},\ell_{3}<\infty$. 
We note that $E\in\Hc(\rot,Q)\cap\H^{1}(Q)$.
Hence, we may evaluate $E$ at the boundary to deduce 
$E_{1}=E_{2}=0$ on $I_{1}\times I_{2}\times\{a_{3}\}$ and 
$E_{1}=E_{3}=0$ on  $I_{1}\times\{a_{2}\}\times I_{3}$. 
Lemma \ref{lem:friedrichs} shows
\begin{align*}
\norm{E_{1}}_{\Lt(Q)}
&\leq 
c_{3}\norm{\na E_{1}}_{\Lt(Q)},
&
\norm{E_{3}}_{\Lt(Q)}
&\leq 
c_{2}\norm{\na E_{3}}_{\Lt(Q)},\\
\norm{E_{2}}_{\Lt(Q)}
&\leq 
c_{3}\norm{\na E_{2}}_{\Lt(Q)},
\end{align*}
and thus
$\norm{E}_{\Lt(Q)}\leq 
\max\{c_{2},c_{3}\}\norm{\na E}_{\Lt(Q)}$.
Finally, \eqref{eq:gaffred} yields the closed range estimate for $\rotc$,
which completes the main part of the proof.

For the remaining part, let $\db_{Q}<2$, and, 
without loss of generality, $\ell_{1}=\ell_{2}=b_{1}=b_{2}=\infty$
and $a_{1},a_{2}\in\{-\infty,0\}$.
Note that, due to the Helmholtz decomposition we have
\begin{align}
\label{eq:helmnoharm}
N(\rotc)^\bot
=\ol{R(\rot)}
=N(\div)\oplus\calH_{\sfD}(Q)
=N(\div),
\end{align}
by Corollary \ref{cor:harmfields}, as $Q$ is convex.
Hence, to contradict the closed range, we shall define sequences 
$$(E_{n})\subset\H^{1}(Q)\cap D(\rotc)\cap N(\div)$$
for either special case.

Let $\db_{Q}=1$, i.e., $\ell_{3}<\infty$. 
We define $(E_{n})$ by $E_{n}(x)\coloneqq f_{n}(x_{1})f_{n}(x_{2})e^{3}$ 
with $f_{n}$ as in \eqref{eq:fn}
and the third unit vector $e^{3}\in\rt$. 
Then $(E_{n})$ is admissible with
$\norm{E_{n}}_{\Lt(Q)}^{2}\sim n^{2}$
and $\norm{\rotc E_{n}}_{\Lt(Q)}^{2}\sim n$.

Let $\db_{Q}=0$, i.e., $\ell_{3}=b_{3}=\infty$ and $a_{3}\in\{-\infty,0\}$.
With $f_{n}$ from \eqref{eq:fn}
we choose $(E_{n})$ by 
$$E_{n}(x)
\coloneqq f_{n}(x_{3})
f_{n}\big(|x'|\big)
\begin{bmatrix}
x_{2}\\
-x_{1}\\
0
\end{bmatrix},
\qquad
x':=
\begin{bmatrix}
x_{1}\\
x_{2}
\end{bmatrix}.$$
Then $(E_{n})$ is admissible with
$$\rot E_{n}(x)=
f_{n}'(x_{3})
f_{n}\big(|x'|\big)
\begin{bmatrix}
x'\\
0
\end{bmatrix}
-f_{n}(x_{3})
\Big(2f_{n}\big(|x'|\big)
+|x'|f_{n}'\big(|x'|\big)\Big)
e^{3}.$$
Hence, using cylinder coordinates, we have
\begin{align*}
\norm{E_{n}}_{\Lt(Q)}^{2}
&=\int_{Q}f_{n}^{2}(x_{3})f_{n}^{2}\big(|x'|\big)|x'|^{2}
\geq\frac{\pi}{2}(n-2)\int_{1}^{n+1}r^{3}f_{n}^{2}(r)\,dr
\geq\frac{\pi}{4}n\int_{2}^{n}r^{3}\,dr
\sim n^{5},\\
\norm{\rot E_{n}}_{\Lt(Q)}^{2}
&\leq16\Big(
\int_{Q}(f_{n}')^{2}(x_{3})f_{n}^{2}\big(|x'|\big)|x'|^{2}
+\int_{Q}f_{n}^{2}(x_{3})f_{n}^{2}\big(|x'|\big)
+\int_{Q}f_{n}^{2}(x_{3})(f_{n}')^{2}\big(|x'|\big)|x'|^{2}
\Big)\\
&\leq32\pi\Big(
2\int_{1}^{n+1}r^{3}f_{n}^{2}(r)\,dr
+n\int_{1}^{n+1}rf_{n}^{2}(r)\,dr
+n\int_{1}^{n+1}r^{3}(f_{n}')^{2}(r)\,dr
\Big)\\
&\leq32\pi\Big(
2\int_{1}^{n+1}r^{3}\,dr
+n\int_{1}^{n+1}r\,dr
+n\int_{1}^{2}r^{3}\,dr
+n\int_{n}^{n+1}r^{3}\,dr
\Big)
\sim n^{4}.
\end{align*}

Thus, in both cases, a closed range estimate for $\rotc$ cannot hold,
and $R(\rotc_{Q})$ is not closed.
\end{proof}

\begin{rem}
\label{rem:helmnoharm}
Let us clarify \eqref{eq:helmnoharm}.
By Corollary \ref{cor:harmfields} there is only the trivial harmonic Dirichlet field,
i.e., $\calH_{\sfD}(Q)=\{0\}$.
Then the projection theorem shows the orthogonal  Helmholtz-type decompositions
\begin{align*}
\Lt(Q)&=\ol{R(\rot)}\oplus N(\rotc),\\
N(\div)&=\ol{R(\rot)}\oplus\calH_{\sfD}(Q)=\ol{R(\rot)}.
\end{align*}
For more detailed results on harmonic fields see, 
e.g., \cite{P1982a,PW2021a}. 
\end{rem}

\begin{theo}[closed range of the divergence]
\label{theo:closedrangediv} 
Let $Q$ be as in \eqref{eq:q}. 
Then:
$$R(\divc_{Q})\text{ closed}
\qequi
R(\na_{Q})\text{ closed}
\qequi
\db_{Q}=3$$
\end{theo}

\begin{proof} 
Again, the first equivalence is a direct consequence of the closed range theorem.  
In case $\db_{Q}=3$, $Q$ is bounded and
the Rellich-Kondrachov selection theorem, i.e.,
the compactness of the embedding $\H^{1}(Q)\cpt\Lt(Q)$,
yields closedness of the range of $R(\na)$. 

For $\db_{Q}<3$ 
we note that $N(\na_{Q})=\{0\}$,
and we have the following counterexamples
using again $f_{n}$ from \eqref{eq:fn}:

If $\db_{Q}=2$, e.g., $\ell_{1},\ell_{2}<\infty$ 
and, without loss of generality, 
$\ell_{3}=b_{3}=\infty$, $a_{3}\in\{-\infty,0\}$,
let $u_{n}(x)\coloneqq f_{n}(x_{3})$. 
Then $\norm{u_{n}}_{\Lt(Q)}^{2}\sim n$
and $\norm{\na u_{n}}_{\Lt(Q)}^{2}\sim 1$.

If $\db_{Q}=1$, e.g., $\ell_{3}<\infty$ 
and, without loss of generality, 
$\ell_{1}=\ell_{2}=b_{1}=b_{2}=\infty$, $a_{1},a_{2}\in\{-\infty,0\}$,
let $u_{n}(x)\coloneqq f_{n}\big(|x'|\big)$.  
Then $\norm{u_{n}}_{\Lt(Q)}^{2}\sim n^{2}$
and $\norm{\na u_{n}}_{\Lt(Q)}^{2}\sim n$.

If $\db_{Q}=0$, e.g., without loss of generality, 
$\ell_{j}=b_{j}=\infty$ and $a_{j}\in\{-\infty,0\}$
for all $j\in\{1,2,3\}$,
let $u_{n}(x)\coloneqq f_{n}\big(|x|\big)$. 
Then $\norm{u_{n}}_{\Lt(Q)}^{2}\sim n^{3}$
and $\norm{\na u_{n}}_{\Lt(Q)}^{2}\sim n^{2}$.

Thus, in any case, a closed range estimate for $\na$ cannot hold,
and $R(\na_{Q})$ is not closed.
\end{proof}

\begin{rem}
\label{rem:closedrangediv}
There are more proofs of Theorem \ref{theo:closedrangediv}.
\begin{itemize}
\item[\bf(i)]
Another approach is direct verification of the Poincar\'e inequality, i.e.,
for all $u\in\H^{1}(Q)$
$$\norm{u-u_{0}}_{\Lt(Q)}\leq 
c\norm{\na u}_{\Lt(Q)},
\qquad
u_{0}\coloneqq\frac{1}{|Q|}\int_{Q}u.$$ 
\item[\bf(ii)]
A third option, showing the full symmetry of our arguments,
is to copy the proof of Theorem \ref{theo:closedrangerot},
now  for $\divc$ with homogeneous normal boundary conditions, implying that
$E_{1}=0$ at $\{a_{1}\}\times I_{2}\times I_{3}$
and $E_{2}=0$ at $I_{1}\times\{a_{2}\}\times I_{3}$
and $E_{3}=0$ at $I_{1}\times I_{2}\times\{a_{3}\}$.
\end{itemize}
\end{rem}

\begin{rem}[Friedrichs'/Poincare estimates]
\label{rem:closedrange}
Let 
$$c_{2,3}\coloneqq\max\{c_{2},c_{3}\},
\qquad
c_{1,2,3}\coloneqq\max\{c_{1},c_{2},c_{3}\}.$$
Then Lemma \ref{lem:friedrichs} and 
small modifications of the proofs of 
Theorem \ref{theo:closedrangerot}, Theorem \ref{theo:closedrangediv},
and Remark \ref{rem:closedrangediv},
show the following:
\begin{itemize}
\item[\bf(i)]
Let $\db_{Q}\geq1$, e.g., $\ell_{3}<\infty$. 
For all $u\in\Hc^{1}(Q)$ it holds
$$\norm{u}_{\Lt(Q)}\leq 
c_{3}\norm{\na u}_{\Lt(Q)}.$$
\item[\bf(ii)]
Let $\db_{Q}\geq2$, e.g., $\ell_{2},\ell_{3}<\infty$. 
For all $E\in\Hc(\rot,Q)\cap\H(\div,Q)$
it holds $E\in\H^{1}(Q)$ and
$$\norm{E}_{\Lt(Q)}\leq 
c_{2,3}\norm{\na E}_{\Lt(Q)}
\leq c_{2,3}\big(\norm{\rotc E}_{\Lt(Q)}^{2} 
+\norm{\div E}_{\Lt(Q)}^{2}\big)^{1/2}.$$
\item[\bf(iii)]
Let $\db_{Q}=3$, i.e., $\ell_{1},\ell_{2},\ell_{3}<\infty$. 
For all $E\in\H(\rot,Q)\cap\Hc(\div,Q)$
it holds $E\in\H^{1}(Q)$ and
$$\norm{E}_{\Lt(Q)}\leq 
c_{1,2,3}\norm{\na E}_{\Lt(Q)}
\leq c_{1,2,3}\big(\norm{\rot E}_{\Lt(Q)}^{2} 
+\norm{\divc E}_{\Lt(Q)}^{2}\big)^{1/2}.$$
\end{itemize}
\end{rem}

\subsection{An Example with Mixed Boundary Conditions}
\label{sec:exmixbc}

Another detailed look into the proof of Theorem \ref{theo:closedrangerot} shows that, 
if mixed boundary conditions are asked for, 
then the respective $\rot$ can be established to have closed range, 
even though only 
$$\db_{Q}=1<2.$$
In order to keep this exposition as focussed as possible, we shall introduce mixed boundary conditions in a rather ad-hoc way. 
We refer to the literature for the proper set-up and the corresponding Hilbert complex structure, 
see, in particular, \cite{PS2022a}.

Let $Q\coloneqq\reals^{2}\times (0,1)$, i.e., $\db_{Q}=1$,
and let $\ga_{0}\coloneqq\reals^{2}\times\{0\}$ and $\ga_{1}\coloneqq\reals^{2}\times\{1\}$. 
With the help of test fields 
$$\Ci_{\ga_{\ell}}(Q)
\coloneqq\big\{\phi|_{Q}:\phi\in\Cic(\rt)\,\wedge\,\dist(\supp\phi,\ga_{\ell})>0\big\}$$
we define restrictions $\rot_{0}$ and $\div_{1}$ of $\rot$ and $\div$ by 
$$D(\rot_{0})\coloneqq\H_{\ga_{0}}(\rot,Q)\coloneqq\ol{\Ci_{\ga_{0}}(Q)}^{\H(\rot,Q)},
\qquad
D(\div_{1})\coloneqq\H_{\ga_{1}}(\div,Q)\coloneqq\ol{\Ci_{\ga_{1}}(Q)}^{\H(\div,Q)},$$
realising mixed homogeneous boundary conditions.
We show that 
$$R(\rot_{0})\text{ is closed.}$$

Indeed, $\rot_{0}$ and $\div_{1}$ are densely defined and closed. 
It follows from \cite{PS2022a} that $\rot_{0}^{*}=\rot_{1}$
and that the complex property also holds for mixed boundary conditions, i.e., 
$$N(\rot_{0})^\bot=\ol{R(\rot_{1})}\subset N(\div_{1}).$$
Thus, again, it suffices to establish Gaffney's inequality for 
$E\in D(\rot_{0})\cap D(\div_{1})$ and to show that 
$E$ satisfies boundary conditions allowing 
for Friedrichs' estimate to hold (Lemma \ref{lem:friedrichs}).  

Let $E\in D(\rot_{0})\cap N(\rot_{0})^{\bot}\subset D(\rot_{0})\cap N(\div_{1})$. 
Let $\phi\in\C^{\infty}\big(\rt,[0,1]\big)$
with bounded derivative
and $\phi=1$ near $\ga_{0}$ and $\phi=0$ near $\ga_{1}$.
Then $E=\phi E+(1-\phi)E$
as well as (by mollification) 
$\phi E\in\Hc(\rot,Q)\cap\H(\div,Q)$
and $(1-\phi)E\in\H(\rot,Q)\cap\Hc(\div,Q)$.
As $Q$ is convex Theorem \ref{theo:gaffneyconvex} yields
$\phi E,\,(1-\phi)E\in\H^{1}(Q)$, that is,
$E\in\H^{1}(Q)$.
Similar to the proof of Lemma \ref{lem:approxG}, let
$\varphi\in\Cic\big(\rt,[0,1]\big)$ such that
$\varphi|_{B(0,1)}=1$ and $\varphi|_{\rt\setminus B(0,2)}=0$,
and put $\varphi_{r}\coloneqq\varphi(\,\cdot\,/r)$ for $r>0$.
Then $\varphi_{r}|_{B(0,r)}=1$
and $\varphi_{r}|_{\rt\setminus B(0,2r)}=0$.
Note that $\supp\na\varphi_{r}\subset\ol{B(0,2r)}\setminus B(0,r)$
and $|\na\varphi_{r}|\leq c/r$.
Lemma \ref{lem:costabel} shows
$$\bnorm{\na(\varphi_{r}E)}_{\L^{2}(Q)}^{2}
=\bnorm{\rot(\varphi_{r}E)}_{\L^{2}(Q)}^{2}
+\bnorm{\div(\varphi_{r}E)}_{\L^{2}(Q)}^{2}$$
(integration just over $Q\cap(-3r,3r)^{3}$,
flat boundaries, and mixed boundary conditions on the particular faces).
Lebesgue's dominated convergence theorem 
together with the product rules yields for $r\to\infty$
\begin{align}
\label{eq:gradrotdivexample}
\norm{\na E}_{\L^{2}(Q)}^{2}
=\norm{\rot E}_{\L^{2}(Q)}^{2}
+\norm{\div E}_{\L^{2}(Q)}^{2},
\end{align}
cf.~\eqref{eq:estr1} and \eqref{eq:estr2}.
The tangential boundary condition implies $E_{1}=E_{2}=0$ at $\ga_{0}$
and the normal boundary condition shows $E_{3}=0$ at $\ga_{1}$.
Thus, in any case, $E_{j}$ satisfies the Friedrichs estimate from Lemma \ref{lem:friedrichs},
and the closed range inequality for $\rot_{0}$ follows by \eqref{eq:gradrotdivexample} and $\div E=0$.

%

Note that by symmetry of $\ga_{0}$ and $\ga_{1}$, 
the closed range theorem, and Lemma \ref{lem:friedrichs},
$R(\div_{1})$ is closed as well since $\div_{1}^{*}=-\na_{0}$.
Therefore, all operators in the de Rham complex with these particular mixed boundary conditions
\begin{equation}
\label{eq:derham3}
\def\arrowlength{12ex}
\def\arrowdistance{.8}
\begin{tikzcd}[column sep=\arrowlength]
\Lt(Q) 
\ar[r, rightarrow, shift left=\arrowdistance, "\A_{0}=\na_{0}"] 
\ar[r, leftarrow, shift right=\arrowdistance, "\A_{0}^{*}=-\div_{1}"']
&
\Lt(Q) 
\ar[r, rightarrow, shift left=\arrowdistance, "\A_{1}=\rot_{0}"] 
\ar[r, leftarrow, shift right=\arrowdistance, "\A_{1}^{*}=\rot_{1}"']
& 
\Lt(Q) 
\arrow[r, rightarrow, shift left=\arrowdistance, "\A_{2}=\div_{0}"] 
\arrow[r, leftarrow, shift right=\arrowdistance, "\A_{2}^{*}=-\na_{1}"']
& 
\Lt(Q),
\end{tikzcd}
\end{equation}
cf.~\eqref{eq:derham1} and \eqref{eq:derham2},
have closed range.

\section{Closed Range Results for Global Lipschitz Domains}
\label{sec:clranlip}

In this section we turn to domains that are not necessarily cubes anymore. 
Let $\Xi\subset\rt$ be open, and let 
$$\Phi:\Xi\to\om\coloneqq\Phi(\Xi)$$ 
be an \textbf{admissible bi-Lipschitz transformation},
cf.~Appendix \ref{sec:transtheo}.
The next theorem asserts that admissible transformations preserve closedness of the range. 

The domains in question are called global Lipschitz domains defined as follows.
We say $\om$ is a \textbf{global Lipschitz domain}, 
if there exists an open cuboid $Q\subset\rt$ and an admissible bi-Lipschitz transformation $\Phi$ 
such that $\Phi(Q)=\om$. 
Correspondingly, we define the \textbf{number of directions of boundedness} by
$\db_{\om}\coloneqq\db_{Q}$.

\begin{theo}[closed range invariance]
\label{theo:clranlip} 
Let $\om,\Xi\subset\rt$ 
be open and let $\Phi\colon\Xi\to\om$ 
be an admissible bi-Lipschitz transformation.
Then
$$R(\rot_{\om})\text{ closed}
\qequi
R(\rot_{\Xi})\text{ closed}.$$
The corresponding results also hold for $R(\na_{\om})$, $R(\div_{\om})$,
and $R(\nac_{\om})$, $R(\rotc_{\om})$, $R(\divc_{\om})$.
\end{theo}

\begin{proof}
Assume that $R(\rot_{\Xi})$ is closed, and 
let $(E_{n})$ in $ D(\rot_{\om})=\H(\rot,\om)$ be a sequence such that
$\rot E_{n}\to F$ in $\Lt(\om)$ for some $F\in\Lt(\om)$.
By Theorem \ref{theo:transtheo} we have
$\tau^{1}_{\Phi}E_{n}\in\H(\rot,\Xi)$
and 
$\rot\tau^{1}_{\Phi}E_{n}=\tau^{2}_{\Phi}\rot E_{n}\to\tau^{2}_{\Phi}F$
in $\Lt(\Xi)$.
As $R(\rot_{\Xi})$ is closed we get
$\tau^{2}_{\Phi}F=\rot H\in R(\rot_{\Xi})$
with $H\in D(\rot_{\Xi})=\H(\rot,\Xi)$.
Then $\tau^{1}_{\Phi^{-1}}H\in\H(\rot,\om)=D(\rot_{\om})$
and $\rot\tau^{1}_{\Phi^{-1}}H=\tau^{2}_{\Phi^{-1}}\rot H=F$
by Theorem \ref{theo:transtheo} 
and thus $F\in R(\rot_{\om})$.
Similarly, we see the corresponding results for $R(\na_{\om})$ and $R(\div_{\om})$. 
The remaining assertions follow analogously.
\end{proof}

It is not difficult to see that $\db_{\om}$ does not depend on the particular choice of $Q$ and $\Phi$. 
In fact, this is due to the fact that bounded intervals and unbounded intervals 
cannot be mapped in a bi-Lipschitz way onto another.

\begin{theo}[main theorem]
\label{theo:main}
Let $\om\subset\rt$ be a global Lipschitz domain.
Then
\begin{align*}
R(\nac_{\om})&\text{ closed}
&&\equi&
R(\div_{\om})&\text{ closed}
&&\equi&
\db_{\om}&\geq1\\
R(\rotc_{\om})&\text{ closed}
&&\equi&
R(\rot_{\om})&\text{ closed}
&&\equi&
\db_{\om}&\geq2\\
R(\divc_{\om})&\text{ closed}
&&\equi&
R(\na_{\om})&\text{ closed}
&&\equi&
\db_{\om}&=3
\end{align*}
\end{theo}

\begin{proof}
The statements follow from Banach's closed range theorem
and the characterisations in Section \ref{sec:clrancubes} 
together with the invariance of closed ranges from Theorem \ref{theo:clranlip}.
\end{proof}

\subsection{Examples}
\label{sec:ex}

We provide some admissible transformations
such that Theorem \ref{theo:main} is applicable.



\begin{ex}[convex bodies]
Let $\om\subset\rt$ be open, bounded, and convex. 
Then there exists a bi-Lipschitz map $\Phi:Q\coloneqq (-1,1)^{3}\to\om$,
cf.~\cite{GHKR2008a},
which can be extended to an admissible bi-Lipschitz transformation.
Hence, $\nac_{\om}$, $\div_{\om}$, $\rotc_{\om}$, $\rot_{\om}$, $\divc_{\om}$, $\na_{\om}$ have closed range.
\end{ex}

\begin{ex}[infinite L-shaped pipe] 
Let $\Phi:Q\to\om=\Phi(Q)$
with
$$Q:=\reals\times(0,1)^{2},
\qquad
\Phi(r,t,s):=
\thvec{r}
{|r|+t}
{t+s}.$$
Then $\det J_{\Phi}(r,t,s)=1$ and $\db_{Q}=2$. 
Hence, $\Phi$ is admissible,
and $\nac_{\om}$, $\div_{\om}$, $\rotc_{\om}$, $\rot_{\om}$ have closed range.
\end{ex}

\begin{figure}[h]
\includegraphics[scale=.095]{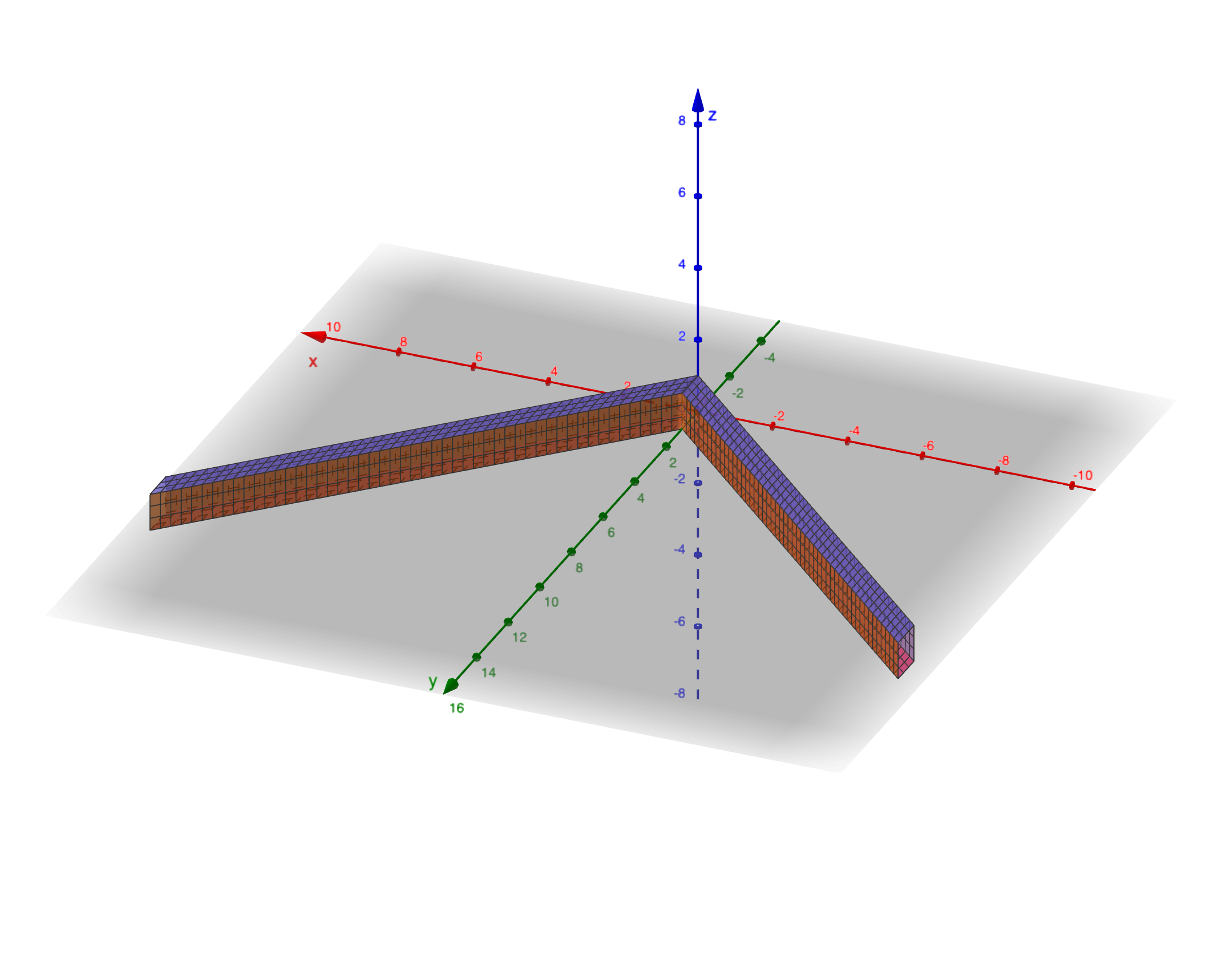}
\includegraphics[scale=.095]{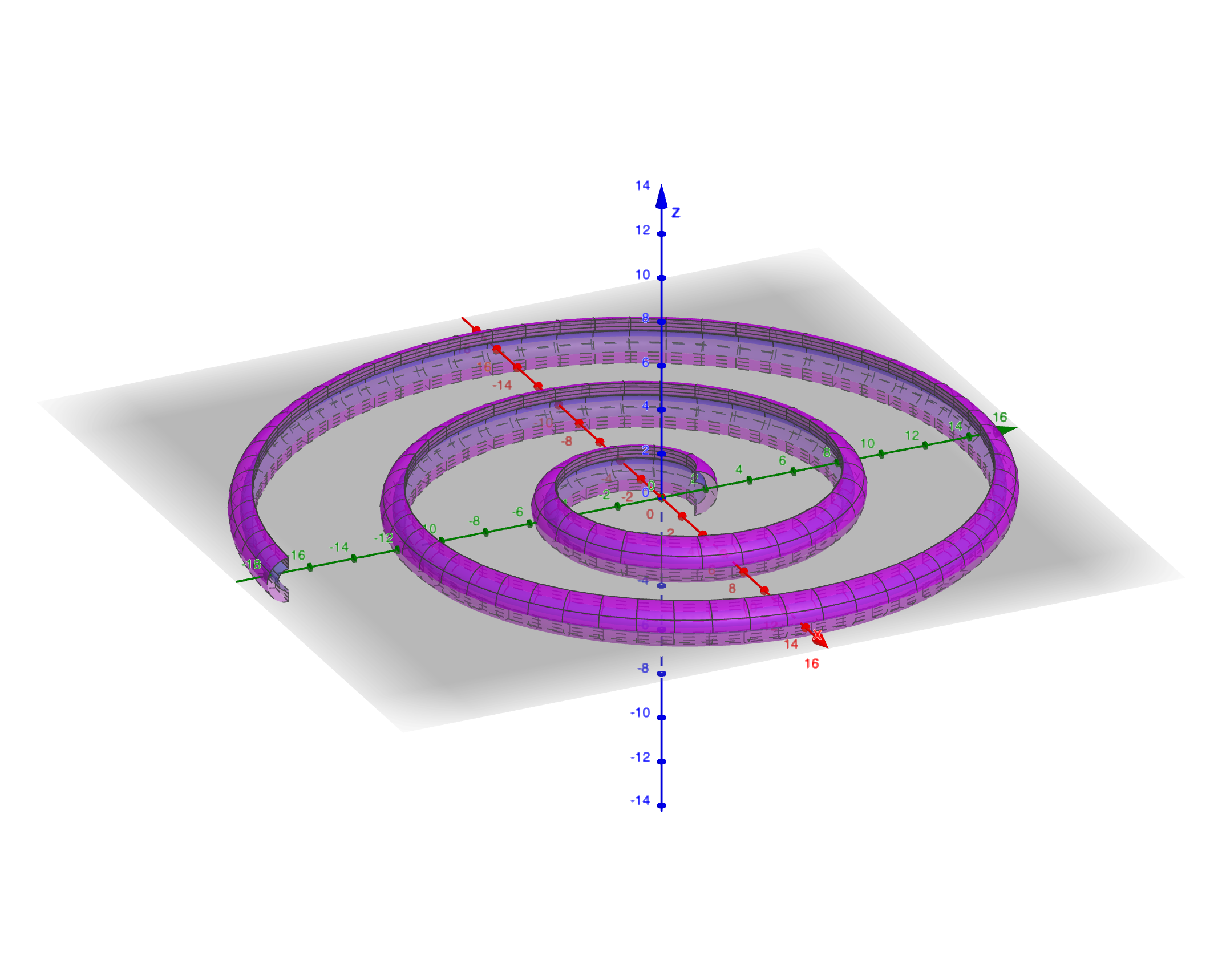}
\caption{plots of the L-shaped pipe and the half snail shell from GeoGebra.org}
\label{pic:geogebra}
\end{figure}

\begin{ex}[infinite growing half snail shell] 
\label{ex:snail}
Let $\Phi:Q\to\om=\Phi(Q)$
with
$$Q:=(2,\infty)\times\big(-\frac{\pi}{2},\frac{\pi}{2}\big)\times\big(\frac{1}{2},1\big)$$
and 
$$\Phi(\varphi,\psi,r):=
\beta(\varphi)
\thvec{\cos\big(\alpha(\varphi)\big)}
{\sin\big(\alpha(\varphi)\big)}
{0}
+
r\gamma(\varphi)
\thvec{\cos\big(\alpha(\varphi)\big)\cos(\psi)}
{\sin\big(\alpha(\varphi)\big)\cos(\psi)}
{\sin(\psi)}.$$
Computing the Jacobian and
$$\det J_{\Phi}(\varphi,\psi,r)
=r\alpha'(\varphi)\gamma^{2}(\varphi)\big(\beta(\varphi)+r\gamma(\varphi)\cos(\psi)\big),$$
we see that the terms
$$\gamma,
\quad
\gamma',
\quad
\alpha'(\varphi)\big(\beta(\varphi)+r\gamma(\varphi)\cos(\psi)\big),
\quad
\beta'(\varphi)+r\gamma'(\varphi)\cos(\psi)$$
have to be bounded to fulfill
$J_{\Phi}\in\L^{\infty}(Q)$ and $\det J_{\Phi}\geq c>0$,
cf.~Appendix \ref{sec:transtheo}.
One possible choice is
$\alpha(\varphi):=\beta(\varphi):=\sqrt{\varphi}$ and $\gamma(\varphi):=1$
with
$$\det J_{\Phi}(\varphi,\psi,r)
=r\frac{\sqrt{\varphi}+r\cos(\psi)}{2\sqrt{\varphi}}
\geq r\geq\frac{1}{4},$$
yielding the picture in Figure \ref{pic:geogebra}.
We have $\db_{Q}=2$.
Hence, $\Phi$ is admissible,
and $\nac_{\om}$, $\div_{\om}$, $\rotc_{\om}$, $\rot_{\om}$ have closed range.
Another possible choice (not producing a nice plot) is
$\alpha(\varphi):=\ln(\varphi)$, $\beta(\varphi):=\varphi$, and $\gamma(\varphi):=1$.
\end{ex}

\section*{Acknowledgments} 

We thank Fabian Elsken-Sicard for 
a very careful reading of the manuscript
and significantly improving the counterexamples 
presented in the proof of Theorem \ref{theo:closedrangerot}
as well as finding some bugs in the proof of Theorem \ref{theo:transtheo}.
We also thank Francesco Ferraresso and Joachim Sch\"oberl
for pointing out mistakes in an earlier version of Example \ref{ex:snail}.



\printbibliography

\vspace*{5mm}
\hrule
\vspace*{3mm}

\appendix

\section{The Lipschitz Transformation Theorem}
\label{sec:transtheo}

Let $\Xi\subset\rt$ be open 
and let $\Phi\in\C^{0,1}(\rt,\rt)$ be such that its restriction to $\Xi$, still denoted by
$$\Phi:\Xi\to\Phi(\Xi)=:\om,$$ 
is bi-Lipschitz and regular, i.e.,
$\Phi\in\C^{0,1}(\ol{\Xi},\ol{\om})$ 
and $\Phi^{-1}\in\C^{0,1}\ol{\om},\ol{\Xi})$ with 
$$J_{\Phi}=\Phi'=(\na\Phi)^{\top}\in\L^{\infty}(\Xi),\qquad
\det J_{\Phi}\geq c>0.$$
Such regular bi-Lipschitz transformations $\Phi$ will be called \textbf{admissible}.
For admissible $\Phi$ the inverse and adjunct matrix of $J_{\Phi}$ 
belong to $\L^{\infty}(\Xi)$ as well and shall be denoted by
$$J_{\Phi}^{-1},\qquad
\adj J_{\Phi}\coloneqq (\det J_{\Phi})J_{\Phi}^{-1},$$
respectively.
We denote the composition with $\Phi$
by tilde, i.e., for any tensor field $v$ we define 
$$\wt{v}\coloneqq v\circ\Phi.$$

We introduce a new notation 
$$\Hd=\Hc\quad\text{or}\quad\Hd=\H$$
to handle spaces with and without boundary conditions simultaneously.

In the following, let $\Phi$ be admissible.

\begin{theo}[transformation theorem]
\label{theo:transtheo}
Let $u\in\Hd^{1}(\om)$, $E\in\Hd(\rot,\om)$, and $H\in\Hd(\div,\om)$.
Then 
\begin{align*}
\tau^{0}_{\Phi}u\coloneqq\wt{u}&\in\Hd^{1}(\Xi)
&&\text{and}&
\na\tau^{0}_{\Phi}u&=\tau^{1}_{\Phi}\na u,\\
\tau^{1}_{\Phi}E\coloneqq J_{\Phi}^{\top}\wt{E}&\in\Hd(\rot,\Xi)
&&\text{and}&
\rot\tau^{1}_{\Phi}E&=\tau^{2}_{\Phi}\rot E,\\
\tau^{2}_{\Phi}H\coloneqq (\adj J_{\Phi})\wt{H}&\in\Hd(\div,\Xi)
&&\text{and}&
\div\tau^{2}_{\Phi}H&=\tau^{3}_{\Phi}\div H
\end{align*}
with $\tau^{3}_{\Phi}f\coloneqq (\det J_{\Phi})\wt{f}=(\det J_{\Phi})\tau^{0}_{\Phi}f\in\Lt(\Xi)$
for $f\in\Lt(\om)$.
Moroever, 
\begin{align*}
\tau^{0}_{\Phi}:\Hd^{1}(\om)&\to\Hd^{1}(\Xi),
&
\tau^{1}_{\Phi}:\Hd(\rot,\om)&\to\Hd(\rot,\Xi),\\
\tau^{3}_{\Phi}:\Lt(\om)&\to\Lt(\Xi),
&
\tau^{2}_{\Phi}:\Hd(\div,\om)&\to\Hd(\div,\Xi)
\end{align*}
are topological isomorphisms with norms depending on $\Xi$
and $J_{\Phi}$ only. The inverse operators and the $\Lt$-adjoints, 
i.e., the Hilbert space adjoints of 
$\tau^{q}_{\Phi}:\Lt(\om)\to\Lt(\Xi)$, $q\in\{0,1,2,3\}$,
are given by
$$(\tau^{q}_{\Phi})^{-1}=\tau^{q}_{\Phi^{-1}},\qquad
(\tau^{q}_{\Phi})^{*}=\tau^{3-q}_{\Phi^{-1}}.$$
\end{theo}

A proof for differential forms
can be found in the appendix of \cite{BPS2019a}.

\begin{proof}
We use Rademacher's theorem for Lipschitz functions, that is, 
any Lipschitz continuous function is differentiable almost everywhere 
with uniformly bounded derivative.

We start with the gradient:
For $u\in\Cd^{0,1}(\om)$
we have by Rademacher's theorem 
$\wt{u}\in\Cd^{0,1}(\Xi)$
and the standard chain rule $(\wt{u})'=\wt{u'}\Phi'$ holds, i.e.,
\begin{align}
\label{transtheo:gradformulaone}
\na\wt{u}=\na\Phi\wt{\na u}=J_{\Phi}^{\top}\wt{\na u}.
\end{align}
For $u\in\Hd^{1}(\om)$ we pick a sequence $(u^{\ell})$ in $\Cd^{0,1}(\om)$ 
with compact supports in $\rt$
such that $u^{\ell}\to u$ in $\Hd^{1}(\om)$. 
Then $\wt{u^{\ell}}\to\wt{u}$ and $\wt{\na u^{\ell}}\to\wt{\na u}$ in $\Lt(\Xi)$
by the standard transformation theorem. 
We have
$(\wt{u^{\ell}})\subset\Cd^{0,1}(\Xi)$
with compact supports in $\rt$
by \eqref{transtheo:gradformulaone} and thus
$\wt{u^{\ell}}\in\Hd^{1}(\Xi)$
with
$$\wt{u^{\ell}}\to\wt{u},\quad
\na\wt{u^{\ell}}=J_{\Phi}^{\top}\wt{\na u^{\ell}}\to J_{\Phi}^{\top}\wt{\na u}\qquad
\text{in }\Lt(\Xi).$$
Since $\dot\na:\Hd^{1}(\Xi)\subset\Lt(\Xi)\to\Lt(\Xi)$ is closed,
we conclude $\wt{u}\in\Hd^{1}(\Xi)$ and 
$$\na\wt{u}=J_{\Phi}^{\top}\wt{\na u}.$$

Next, we consider the $\rot$-operator:
For this, let $E\in\Cd^{0,1}(\om)$
with compact support in $\rt$. 
Then $\wt{E}\in\Cd^{0,1}(\Xi)$
with compact support in $\rt$ and
$$J_{\Phi}^{\top}\wt{E}=\na\Phi\wt{E}
=[\na\Phi_{1}\;\na\Phi_{2}\;\na\Phi_{3}]\wt{E}
=\sum_{n}\wt{E}_{n}\na\Phi_{n}.$$
As 
$\na\Phi_{n}
\in R(\na_{\Xi_{\mathsf{b}}})
\subset N(\rot_{\Xi_{\mathsf{b}}})
\subset\H(\rot,\Xi_{\mathsf{b}})$ 
for any open and bounded subset $\Xi_{\mathsf{b}}\subset\Xi$,
e.g., $\ol{\Xi_{\mathsf{b}}}=\supp E$,
we conclude
$J_{\Phi}^{\top}\wt{E}\in\H(\rot,\Xi)$ and also
$J_{\Phi}^{\top}\wt{E}\in\Hd(\rot,\Xi)$ by mollification 
as well as (by the previous result for $\na$)
\begin{align}
\nonumber
\rot(J_{\Phi}^{\top}\wt{E})
&=\sum_{n}\na\wt{E}_{n}\times\na\Phi_{n}
=\sum_{n}(J_{\Phi}^{\top}\wt{\na E_{n}})\times\na\Phi_{n}\\
\label{transtheo:rotformulaLipschitz}
&=\sum_{n}\big([\na\Phi_{1}\;\na\Phi_{2}\;\na\Phi_{3}]\,\wt{\na E_{n}}\big)\times\na\Phi_{n}\\
\nonumber
&=\sum_{n,m}\wt{\p_{m}E_{n}}\na\Phi_{m}\times\na\Phi_{n}
=\sum_{n<m}(\wt{\p_{m}E_{n}}-\wt{\p_{n}E_{m}})\na\Phi_{m}\times\na\Phi_{n}\\
\nonumber
&=[\na\Phi_{2}\times\na\Phi_{3}\quad\na\Phi_{3}\times\na\Phi_{1}\quad\na\Phi_{1}\times\na\Phi_{2}]\,\wt{\rot E}
=(\adj J_{\Phi})\,\wt{\rot E}.
\end{align}
For $E\in\Hd(\rot,\om)$ we pick a sequence $(E^{\ell})$ in $\Cd^{0,1}(\om)$ 
with compact supports in $\rt$
such that $E^{\ell}\to E$ in $\H(\rot,\om)$. 
Then $\wt{E^{\ell}}\to\wt{E}$ and $\wt{\rot E^{\ell}}\to\wt{\rot E}$ in $\Lt(\Xi)$.
Hence by \eqref{transtheo:rotformulaLipschitz}
$J_{\Phi}^{\top}\wt{E^{\ell}}\in\Hd(\rot,\Xi)$ with
$$J_{\Phi}^{\top}\wt{E^{\ell}}\to J_{\Phi}^{\top}\wt{E},\quad
\rot(J_{\Phi}^{\top}\wt{E^{\ell}})=(\adj J_{\Phi})\,\wt{\rot E^{\ell}}\to(\adj J_{\Phi})\,\wt{\rot E}\qquad
\text{in }\Lt(\Xi).$$
Since $\dot\rot:\Hd(\rot,\Xi)\subset\Lt(\Xi)\to\Lt(\Xi)$ is closed,
we conclude $J_{\Phi}^{\top}\wt{E}\in\Hd(\rot,\Xi)$ and 
$$\rot(J_{\Phi}^{\top}\wt{E})=(\adj J_{\Phi})\,\wt{\rot E}.$$

For the divergence,
let $H\in\Cd^{0,1}(\om)$
with compact support in $\rt$. 
Then $\wt{H}\in\Cd^{0,1}(\Xi)$
with compact support in $\rt$ and
\begin{align*}
(\adj J_{\Phi})\wt{H}
&=[\na\Phi_{2}\times\na\Phi_{3}\quad\na\Phi_{3}\times\na\Phi_{1}\quad\na\Phi_{1}\times\na\Phi_{2}]\wt{H}
=\sum_{(n,m,l)}\wt{H}_{n}\na\Phi_{m}\times\na\Phi_{l},
\end{align*}
cf.~\eqref{transtheo:rotformulaLipschitz},
where the summation is over the three even permutations $(n,m,l)$ of $(1,2,3)$.
As we have 
$\na\Phi_{m}\times\na\Phi_{l}
=\rot(\Phi_{m}\na\Phi_{l})
\in R(\rot_{\Xi_{\mathsf{b}}})
\subset N(\div_{\Xi_{\mathsf{b}}})
\subset\H(\div,\Xi_{\mathsf{b}})$,
with $\Xi_{\mathsf{b}}$ as before,
we conclude 
$(\adj J_{\Phi})\wt{H}\in\H(\div,\Xi)$ and thus also 
$(\adj J_{\Phi})\wt{H}\in\Hd(\div,\Xi)$ by mollification as well as 
\begin{align}
\nonumber
\div\big((\adj J_{\Phi})\wt{H}\big)
&=\sum_{(n,m,l)}\na\wt{H}_{n}\cdot(\na\Phi_{m}\times\na\Phi_{l})
=\sum_{(n,m,l)}(J_{\Phi}^{\top}\wt{\na H_{n}})\cdot(\na\Phi_{m}\times\na\Phi_{l})\\
\label{transtheo:divformulaLipschitz}
&=\sum_{(n,m,l)}\big([\na\Phi_{1}\;\na\Phi_{2}\;\na\Phi_{3}]\,\wt{\na H_{n}}\big)\cdot(\na\Phi_{m}\times\na\Phi_{l})\\
\nonumber
&=\sum_{(n,m,l),k}\wt{\p_{k}H_{n}}\na\Phi_{k}\cdot(\na\Phi_{m}\times\na\Phi_{l})\\
\nonumber
&\overset{k=n}{=}(\det\na\Phi)\,\wt{\div H}
=(\det J_{\Phi})\,\wt{\div H}.
\end{align}
For $H\in\Hd(\div,\om)$ let $(H^{\ell})\subset\Cd^{0,1}(\om)$
be a sequence with compact supports in $\rt$
and $H^{\ell}\to H$ in $\H(\div,\om)$. 
Then $\wt{H^{\ell}}\to\wt{H}$ and $\wt{\div H^{\ell}}\to\wt{\div H}$ in $\Lt(\Xi)$. 
By \eqref{transtheo:divformulaLipschitz}
$(\adj J_{\Phi})\wt{H^{\ell}}\in\Hd(\div,\Xi)$ and also
$(\adj J_{\Phi})\wt{H^{\ell}}\to(\adj J_{\Phi})\wt{H}$ and 
$\div\big((\adj J_{\Phi})\wt{H^{\ell}}\big)=(\det J_{\Phi})\wt{\div H^{\ell}}\to(\det J_{\Phi})\wt{\div H}$
in $\Lt(\Xi)$.
Since $\dot\div:\Hd(\div,\Xi)\subset\Lt(\Xi)\to\Lt(\Xi)$ is closed,
we conclude that $(\adj J_{\Phi})\wt{H}\in\Hd(\div,\Xi)$ and 
$$\div\big((\adj J_{\Phi})\wt{H}\big)=(\det J_{\Phi})\wt{\div H}.$$

The statements on the topological isomorphisms follow by symmetry in $\Xi$ and $\om$.

Finally, concerning the inverse operators and $\Lt$-adjoints we consider, e.g., $q=1$.
Then using $J_{\Phi^{-1}}=J_{\Phi}^{-1}\circ\Phi^{-1}$ we compute 
$$\tau^{1}_{\Phi^{-1}}\tau^{1}_{\Phi}E
=\tau^{1}_{\Phi^{-1}}J_{\Phi}^{\top}\wt{E}
=J_{\Phi^{-1}}^{\top}\big((J_{\Phi}^{\top}\wt{E})\circ\Phi^{-1}\big)
=\big(J_{\Phi}^{-\top}J_{\Phi}^{\top}\wt{E})\circ\Phi^{-1}
=E,$$
i.e., $(\tau^{1}_{\Phi})^{-1}=\tau^{1}_{\Phi^{-1}}$, and
\begin{align*}
\scp{\tau^{1}_{\Phi}E}{\Psi}_{\Lt(\Xi)}
=\scp{J_{\Phi}^{\top}\wt{E}}{\Psi}_{\Lt(\Xi)}
&=\bscp{E}{(\det J_{\Phi^{-1}})(J_{\Phi}\Psi)\circ\Phi^{-1}}_{\Lt(\om)}\\
&=\bscp{E}{(\det J_{\Phi^{-1}})J_{\Phi^{-1}}^{-1}(\Psi\circ\Phi^{-1})}_{\Lt(\om)}\\
&=\bscp{E}{(\adj J_{\Phi^{-1}})(\Psi\circ\Phi^{-1})}_{\Lt(\om)}
=\scp{E}{\tau^{2}_{\Phi^{-1}}\Psi}_{\Lt(\om)},
\end{align*}
i.e., $(\tau^{1}_{\Phi})^{*}=\tau^{2}_{\Phi^{-1}}$.
\end{proof}

\begin{rem}[transformation theorem]
\label{rem:transtheo}
More explicitly, Theorem \ref{theo:transtheo} shows
\begin{align*}
\forall\,u&\in\Hd^{1}(\om)
&
\wt{u}&\in\Hd^{1}(\Xi)
&&\text{and}&
\na\wt{u}
&=J_{\Phi}^{\top}\wt{\na u},\\
\forall\,E&\in\Hd(\rot,\om)
&
J_{\Phi}^{\top}\wt{E}&\in\Hd(\rot,\Xi)
&&\text{and}&
\rot(J_{\Phi}^{\top}\wt{E})
&=(\adj J_{\Phi})\,\wt{\rot E},\\
\forall\,H&\in\Hd(\div,\om)
&
(\adj J_{\Phi})\wt{H}&\in\Hd(\div,\Xi)
&&\text{and}&
\div\big((\adj J_{\Phi})\,\wt{H}\big)
&=(\det J_{\Phi})\,\wt{\div H},\\
\forall\,E&\in\eps^{-1}\Hd(\div,\om)
&
\eps_{\Phi}J_{\Phi}^{\top}\wt{E}&\in\Hd(\div,\Xi)
&&\text{and}&
\div(\eps_{\Phi}J_{\Phi}^{\top}\wt{E})
&=(\det J_{\Phi})\,\wt{\div\eps E},
\end{align*}
where
$\eps_{\Phi}
=(\adj J_{\Phi})\wt{\eps}J_{\Phi}^{-\top}$
and $\eps$ is a real, bounded, symmetric, and positive matrix field.
Moreover, 
\begin{align*}
\tau^{1}_{\Phi}:\Hc(\rot,\om)\cap\eps^{-1}\H(\div,\om)
&\to\Hc(\rot,\Xi)\cap\eps_{\Phi}^{-1}\H(\div,\Xi),\\
\quad\tau^{1}_{\Phi}:\H(\rot,\om)\cap\eps^{-1}\Hc(\div,\om)
&\to\H(\rot,\Xi)\cap\eps_{\Phi}^{-1}\Hc(\div,\Xi)
\end{align*}
are topological isomorphisms with norms depending on $\Xi$, $\eps$,
and $J_{\Phi}$ only, and inverses $\tau^{1}_{\Phi^{-1}}$.
\end{rem}

\section{Proof of the Remaining Part in Theorem \ref{theo:gaffneyconvexbd}}
\label{sec:detproofgaffneyconvex}

Let 
$E\in\H(\rot,\om)\cap\Hc(\div,\om)$:
We pick $\om_{\ell}\subset\om$ as before.
For $\om_{\ell}$ we find
$u_{\ell}\in\H^{1}(\om_{\ell})$ 
such that for all $\psi\in\H^{1}(\om_{\ell})$
\begin{align}
\label{pdedefomn2}
\scp{u_{\ell}}{\psi}_{\H^{1}(\om_{\ell})}
=\scp{\div E}{\psi}_{\Lt(\om_{\ell})}
+\scp{E}{\na\psi}_{\Lt(\om_{\ell})}
\end{align}
(Riesz isometry). 
As $\scp{u_{\ell}}{\psi}_{\H^{1}(\om_{\ell})}
=\scp{u_{\ell}}{\psi}_{\Lt(\om_{\ell})}
+\scp{\na u_{\ell}}{\na\psi}_{\Lt(\om_{\ell})}$,
we have
\begin{align*}
\scp{E-\na u_{\ell}}{\na\psi}_{\Lt(\om_{\ell})}
=\scp{u_{\ell}-\div E}{\psi}_{\Lt(\om_{\ell})}
\end{align*}
for all $\psi\in\H^{1}(\om_{\ell})$, i.e.,
$E_{\ell}\coloneqq E-\na u_{\ell}\in\Hc(\div,\om_{\ell})$ 
and $\div E_{\ell}=\div E-u_{\ell}$.
Moreover, $E_{\ell}\in\H(\rot,\om_{\ell})$ 
with $\rot E_{\ell}=\rot E$.
By Lemma \ref{lem:gaffneyconvexsmooth}
we have $E_{\ell}\in\H^{1}(\om_{\ell})$ with
\begin{align}
\label{naomeganest2}
\norm{\na E_{\ell}}_{\Lt(\om_{\ell})}^{2}
&\leq\norm{\rot E_{\ell}}_{\Lt(\om_{\ell})}^{2}
+\norm{\div E_{\ell}}_{\Lt(\om_{\ell})}^{2}
=\norm{\rot E}_{\Lt(\om_{\ell})}^{2}
+\norm{\div E-u_{\ell}}_{\Lt(\om_{\ell})}^{2}.
\end{align}
By setting $\psi=u_{\ell}$ in \eqref{pdedefomn2} we see
\begin{align}
\label{zetanest2}
\norm{u_{\ell}}_{\H^{1}(\om_{\ell})}^{2}
&=\scp{\div E}{u_{\ell}}_{\Lt(\om_{\ell})}
+\scp{E}{\na u_{\ell}}_{\Lt(\om_{\ell})}
\leq\norm{E}_{\H(\div,\om_{\ell})}\norm{u_{\ell}}_{\H^{1}(\om_{\ell})}
\end{align}
and thus
\begin{align}
\label{zetanesttwo2}
\norm{u_{\ell}}_{\H^{1}(,\om_{\ell})}
&\leq\norm{E}_{\H(\div,\om_{\ell})}
\leq\norm{E}_{\H(\div,\om)}.
\end{align}
Combining \eqref{naomeganest2} and the equation part of \eqref{zetanest2} we observe
\begin{align*}
\norm{E_{\ell}}_{\H^{1}(\om_{\ell})}^{2}
&=\norm{E_{\ell}}_{\Lt(\om_{\ell})}^{2}
+\norm{\na E_{\ell}}_{\Lt(\om_{\ell})}^{2}\\
&\leq\norm{E_{\ell}}_{\Lt(\om_{\ell})}^{2}
+\norm{\div E-u_{\ell}}_{\Lt(\om_{\ell})}^{2}
+\norm{\rot E}_{\Lt(\om_{\ell})}^{2}\\
&=\norm{E}_{\Lt(\om_{\ell})}^{2}
+\norm{\na u_{\ell}}_{\Lt(\om_{\ell})}^{2}
+\norm{\div E}_{\Lt(\om_{\ell})}^{2}
+\norm{u_{\ell}}_{\Lt(\om_{\ell})}^{2}
+\norm{\rot E}_{\Lt(\om_{\ell})}^{2}\\
&\qquad\qquad
-2\scp{E}{\na u_{\ell}}_{\Lt(\om_{\ell})}
-2\scp{\div E}{u_{\ell}}_{\Lt(\om_{\ell})}\\
&=\norm{E}_{\H(\rot,\om_{\ell})\cap\H(\div,\om_{\ell})}^{2}
-\norm{u_{\ell}}_{\H^{1}(,\om_{\ell})}^{2},
\end{align*}
and therefore
\begin{align}
\label{hoomeganest2}
\norm{E_{\ell}}_{\H^{1}(\om_{\ell})}
\leq\norm{E}_{\H(\rot,\om_{\ell})\cap\H(\div,\om_{\ell})}
\leq\norm{E}_{\H(\rot,\om)\cap\H(\div,\om)}.
\end{align}

Again, let us denote the extension by zero to $\om$ by $\wt{\cdot}$.
Then by \eqref{zetanesttwo2} and \eqref{hoomeganest2}
the sequences $(\wt{u}_{\ell})$, $(\wt{\na u}_{\ell})$,
and $(\wt{E}_{\ell})$, $(\wt{\na E}_{\ell})$
are bounded in $\Lt(\om)$, 
and we can extract weakly converging subsequences, again denoted by the index $\ell$, such that
\begin{align*}
\wt{u}_{\ell}
&\wto{\Lt(\om)}u\in\Lt(\om),
&
\wt{E}_{\ell}
&\wto{\Lt(\om)}\wh{E}\in\Lt(\om),\\
(\wt{\na u}_{\ell})
&\wto{\Lt(\om)}F\in\Lt(\om),
&
\wt{\na E}_{\ell}
&\wto{\Lt(\om)}G\in\Lt(\om).
\end{align*}
As before, we get $\wh{E}\in\H^{1}(\om)$ and $\na\wh{E}=G$.
For $\Psi\in\Cic(\om)$
with $\Psi\in\Cic(\om_{\ell})$ for $\ell$ large enough
we compute
\begin{align*}
-\scp{\wt{\na u}_{\ell}}{\Psi}_{\Lt(\om)}=-\scp{\na u_{\ell}}{\Psi}_{\Lt(\om_{\ell})}
=\scp{u_{\ell}}{\div\Psi}_{\Lt(\om_{\ell})}
=\scp{\wt{u}_{\ell}}{\div\Psi}_{\Lt(\om)}.
\end{align*}
Letting $\ell\to\infty$ on the left and the right, we obtain
$$-\scp{F}{\Psi}_{\Lt(\om)}=\scp{u}{\div\Psi}_{\Lt(\om)},$$
showing $u\in\H^{1}(\om)$ and $\na u=F$.
Moreover, for $\psi\in\H^{1}(\om)\subset\H^{1}(\om_{\ell})$
we have 
\begin{align*}
\scp{u_{\ell}}{\psi}_{\H^{1}(\om_{\ell})}
=\scp{\wt{u}_{\ell}}{\psi}_{\Lt(\om)}
+\scp{\wt{\na u}_{\ell}}{\na\psi}_{\Lt(\om)}
\to\scp{u}{\psi}_{\H^{1}(\om)}.
\end{align*}
and, by \eqref{pdedefomn2}, we further get
$$\scp{u_{\ell}}{\psi}_{\H^{1}(\om_{\ell})}=\scp{\div E}{\psi}_{\Lt(\om_{\ell})}
+\scp{E}{\na\psi}_{\Lt(\om_{\ell})}
\to\scp{\div E}{\psi}_{\Lt(\om)}
+\scp{E}{\na\psi}_{\Lt(\om)}
=0$$
as $E\in\Hc(\div,\om)$, where the last convergence follows by Lebesgue's dominated convergence theorem.
For $\psi=u$ we get $\norm{u}_{\H^{1}(\om)}=0$, i.e., $u=0$.
Furthermore, we observe that 
$$\scp{\wh{E}}{\wt{E}_{\ell}}_{\Lt(\om)}
+\scp{\na\wh{E}}{\wt{\na E}_{\ell}}_{\Lt(\om)}
\to\scp{\wh{E}}{\wh{E}}_{\Lt(\om)}
+\scp{\na\wh{E}}{\na\wh{E}}_{\Lt(\om)}
=\norm{\wh{E}}_{\H^{1}(\om)}^{2}$$
and
\begin{align*}
\scp{\wh{E}}{\wt{E}_{\ell}}_{\Lt(\om)}
+\scp{\na\wh{E}}{\wt{\na E}_{\ell}}_{\Lt(\om)}& 
=\scp{\wh{E}}{E_{\ell}}_{\Lt(\om_{\ell})}
+\scp{\na\wh{E}}{\na E_{\ell}}_{\Lt(\om_{\ell})}\\
&\leq\norm{\wh{E}}_{\H^{1}(\om_{\ell})}
\norm{E_{\ell}}_{\H^{1}(\om_{\ell})}
\leq\norm{\wh{E}}_{\H^{1}(\om)}
\norm{E}_{\H(\rot,\om)\cap\H(\div,\om)},
\end{align*}
showing
\begin{align}
\label{hoomnormomegahat2}
\norm{\wh{E}}_{\H^{1}(\om)}
\leq\norm{E}_{\H(\rot,\om)\cap\H(\div,\om)}.
\end{align}
Finally, we have 
$E=E_{\ell}+\na u_{\ell}$ in $\om_{\ell}$, i.e., in $\om$
$$\chi_{\om_{\ell}}E
=\wt{E}_{\ell}+\wt{\na u}_{\ell}
\wto{\Lt(\om)}\wh{E}+\na u
=\wh{E}.$$
On the other hand, by Lebesgue's dominated convergence theorem we see
$\chi_{\om_{\ell}}E\to E$ in $\Lt(\om)$. Thus
$E=\wh{E}\in\H^{1}(\om)$ and by \eqref{hoomnormomegahat2}
$$\norm{E}_{\H^{1}(\om)}
\leq\norm{E}_{\H(\rot,\om)\cap\H(\div,\om)},$$
in particular,
$\norm{\na E}_{\Lt(\om)}^{2}
\leq\norm{\rot E}_{\Lt(\om)}^{2}
+\norm{\div E}_{\Lt(\om)}^{2}$ 
by letting $\ell\to\infty$ in \eqref{naomeganest2}.
\hfill$\square$

\end{document}